\documentclass[11pt]{article}

\usepackage{a4}
\usepackage{geometry}
\usepackage{amsfonts}
\usepackage{graphicx}
\usepackage{epstopdf}
\usepackage{amsmath}
\usepackage{amsthm}
\usepackage{stmaryrd}
\usepackage{bm}
\usepackage{hyperref}

\ifpdf
  \DeclareGraphicsExtensions{.eps,.pdf,.png,.jpg}
\else
  \DeclareGraphicsExtensions{.eps}
\fi

 \newtheorem{theorem}{Theorem}[section]

 \theoremstyle{definition}
 
 \theoremstyle{remark}
 \newtheorem{remark}[theorem]{Remark}
 \numberwithin{equation}{section}

\title{A Ternary Cahn--Hilliard Navier--Stokes Model for Two Phase Flow with Precipitation and Dissolution\thanks{Acknowledgment:
Funded by the Deutsche Forschungsgemeinschaft (DFG, German
Research Foundation) -- Project Number 327154368 -- SFB 1313.}}

\author{Christian Rohde
\thanks{Institute of Applied Analysis and Numerical Simulation,
		University of Stuttgart,
		Pfaffenwaldring 57,
		70569 Stuttgart, 
		Germany
  (\mbox{christian.rohde@mathematik.uni-stuttgart.de}, \mbox{lars.von-wolff@mathematik.uni-stuttgart.de}).
 }
\and Lars von Wolff\, \footnotemark[2]  
}

\usepackage{amsopn}

\usepackage{tikz}
\usepackage{cases}
\usetikzlibrary{arrows.meta}

\newcommand{\NN}{{\mathbb{N}}}         
       
\newcommand{\RR}{{\mathbb{R}}}

\newcommand{\SSa}{{\mathcal{S}}}

\newcommand{\vv}{{\mathbf{v}}}
\newcommand{\Phiv}{{\mathbf{\Phi}}}
\newcommand{\nv}{{\mathbf{n}}}
\newcommand{\Jv}{{\mathbf{J}}}
\newcommand{\Sv}{{\mathbf{S}}}
\newcommand{\Av}{{\mathbf{A}}}
\newcommand{\Bv}{{\mathbf{B}}}
\newcommand{\ev}{{\mathbf{e}}}
\newcommand{\xv}{{\mathbf{x}}}

\newcommand{\sv}{{\mathbf{s}}}

\newcommand{\Uv}{{\mathbf{U}}}
\newcommand{\Cv}{{\mathbf{C}}}
\newcommand{\tauv}{{\bm{\tau}}}

\newcommand{\diff}{\,d}											
\newcommand{\eps}{{\varepsilon}}								
\newcommand{\set}[1]{\left\{#1\right\}}            				
\newcommand{\dbdot}{\mathbin{:}}

\newcommand{\jump}[1]{\llbracket #1 \rrbracket}

\hypersetup{
  pdftitle={A Ternary Cahn--Hilliard Navier--Stokes Model for Two Phase Flow with Precipitation and Dissolution},
  pdfauthor={C. Rohde, L. von Wolff},
  pdfkeywords={Fluid flow with reactive transport; Precipitation/dissolution;  Phase field modelling; Asymptotic analysis}
  pdfsubject={35R35, 35Q35, 76D05, 76T99, 76D45, 35C20}
}

\begin{document}

\maketitle

\begin{abstract} \noindent We consider the incompressible  flow of two immiscible fluids in the presence of a solid phase that undergoes changes in time  due to
    precipitation and dissolution effects.  Based on
a seminal sharp interface model a phase field approach is suggested that couples the  Navier-Stokes equations and the solid's ion concentration
transport equation  with the Cahn-Hilliard evolution for the phase fields.\\ The model is shown to preserve  the fundamental conservation constraints and to obey the second law of
thermodynamics for a novel free energy formulation. An extended  analysis for vanishing interfacial width reveals  that in this
limit the sharp interface model is recovered, including all
relevant transmission conditions. Notably, the new phase field model is able to realize Navier-slip conditions for solid-fluid interfaces
in the limit.
\end{abstract}
\providecommand{\keywords}[1]{{\textit{Key words:}} #1\\ \\}
\providecommand{\class}[1]{{\textit{AMS subject classifications:}} #1}
\class{35R35, 35Q35, 76D05, 76T99, 76D45, 35C20}

\section{Introduction}

Multi-phase flow and reactive transport processes are commonly encountered in engineering applications, getting  particularly important in the
context of porous media flow.  Examples comprise processes like concrete carbonation, geological $CO_2$--sequestration involving
 calcite precipitation, ion exchange  in fuel cells or the spreading  of biofilms  in the soil's vadose zone. While the modeling of multi-phase flow is challenging in itself, these applications are even more complex as the involved solid phase can alter the porous medium skeleton in time which in turn
 changes the overall flow dynamics. 
 
In this contribution  we will propose and analyze a mathematical model that governs the incompressible  flow of two immiscible fluids
that interact with each other and a third solid phase composed of a pure mineral material.
This mineral  is supposed to be solvable in exactly one of the fluid phases. We will account  for the process
of precipitation  enlarging the domain occupied by the solid phase, as well as dissolution transferring  solid material
to the fluid phase.
For a  pertinent example  one might think of a mixture of water, oil and natrium chloride, the latter being present as solid, and resolved in water only.

There are multiple approaches to model  evolving interfaces of types fluid-fluid, reactive fluid-solid and  nonreactive fluid-solid  as  encountered in our  multi-phase flow scenario. Physically mostly well-grounded is the sharp interface formulation. The interfaces are represented as codimension-$1$
manifolds, moving according  to their normal velocity. The normal velocity  is determined from transmission conditions that connect to bulk models valid in the respective phases.
An alternative approach is based on phase field modelling. Then, the interface is modeled as a diffuse transition zone of small
 width. Additional order parameters are introduced, approximating the indicator function of each phase in a smooth way. The  evolution equations for the phase field  are combined with the  governing systems for
 the physical quantities like fluid velocity.  A typical requirement on phase field models is thermodynamical
 consistency which  can be achieved  if the  entire set of evolution equations can be understood as  the gradient flow of a free energy functional.
    The free energy functional  is composed  of a bulk free energy, with a minimum for each of the pure phases, and an interfacial energy penalizing large gradients in the phase fields.  The width of the transition zone is controlled by the phase field parameter. If it tends to zero the phase field model should recover the underlying sharp interface approach. The complexity of such phase field models  excludes rigorous treatment but the formal technique of   matched asymptotic expansions can be utilized  to  justify the phase field approach  as an approximation to a sharp interface formulation, see~Ref.~\cite{Caginalp88}.

  The major contribution of this paper is a new phase field model that describes  the motion of two fluidic and a solid phase as described above. Up
  to our knowledge no such phase field model has been proposed before.\\
  First, we explain the underlying sharp interface ansatz in Section \ref{ChapterSI} that fixes the transmission condition between the bulk phases
  via conservation constraints, reactive mass exchange and the interfaces' curvature influence. Notably,  the model
   incorporates a Navier-slip condition at the  fluid-solid  interfaces.  Without the slip condition, classical results\cite{Solonnikov82} show that the sharp interface model would not be well posed.
   
  The phase field model itself, named $\delta$-$2f1s$-model, will be derived in Section  \ref{ChapterDI}, see equations \eqref{rmodel1}--\eqref{rmodel5}.   By construction, solutions  of the    $\delta$-$2f1s$-model will
  obey the physical constraints of total mass, volume fraction  and  ion concentration conservation. Introducing a new free energy function
 it is  proven  that classical solutions of the
  phase field model obey the second law of thermodynamics (see  Theorem \ref{thrmThermo}).
    This is in contrast to previously suggested  phase field models in the  area  of  reactive transport
  (see Ref.~\cite{carina19pre,magnus}) that  lack
  such thermodynamical consistency. The result of Theorem \ref{thrmThermo} relies on the construction of a free energy that  explicitly
  accounts for  the ion concentration and in turn fixes the kinetic reactions at the solid-fluid interfaces. We illustrate the capabilities of the
  $\delta$-$2f1s$-model by a numerical experiment on a channel flow problem and relate it for simplified scenarios to previously suggested phase field models
  in Sections \ref{sNumerics}, \ref{sAlgebra}, respectively. 
  
  To validate our model we investigate  the sharp interface limit in
  Section \ref{ChapterSIL} using  matched asymptotic expansions. The analysis identifies all binary transmission conditions  (and bulk equations) as proposed for the sharp interface ansatz in Section
  \ref{ChapterSI}. Notably, this includes the Navier-slip condition as presented in Section \ref{sSlip}. This result appears to new, not only for ternary
  mixtures but also in the fundamental  context of binary fluid-solid interfaces.

We conclude this introduction relating  the $\delta$-$2f1s$-model to  existing phase field models for incompressible flow problems.
The most commonly used approach  for two-phase flow is to couple the incompressible Navier--Stokes equation with
the Cahn--Hilliard phase field equation. The basic model, called "Model H", was presented by Hohenberg\&Halperin\cite{hohenberg}. From there on a variety
of refined models  has been proposed. An important aspect for us  is the handling of  fluids with different densities. Because the mass
 averaged generalizations proposed by Lowengrub\&Truskinovsky\cite{lowengrub} lead to a non divergence-free vector field, we base our work on the volume-averaged model of
  Abels et al.\cite{agg}.
For a generalization to three fluid phases, Boyer et al.\cite{Boyer06,Boyer10} introduced consistency principles that lead to particular choices of the bulk free energy. Based thereon models for more than three fluid phases have been proposed in e.g. Ref.~\cite{Boyer14,stinner18pre}.
When considering more than two phases, three interfaces can meet at a triple junction. Analysis of this triple junction\cite{Bronsard1993,Garcke98,stinner18pre} shows that the free energy functional implies a contact angle condition between the three interfaces.

For the description of a fluid-solid interface with a phase field model two main ideas can be pursued. Using a model
 for two fluid phases, one can introduce a solid phase as a fluid with very high viscosity  like in Ref.~\cite{Anderson00}.
 In contrast we follow the work of Beckermann\cite{Beckermann} (but see also Ref.~\cite{Beckermann04,Jeong01}), who assigns
 to  the solid a zero-velocity  and solves the flow equations only in the volume fraction occupied by fluid.
Van Noorden\&Eck\cite{noorden2011} incorporated a kinetic reaction at the phase boundary. Based on the more general Diffuse Domain Approach\cite{dda}, Redeker et al.\cite{magnus} proposed a model for precipitation and dissolution in our context, that is one solid and two fluid phases. Both works only consider diffusion in the fluid phase, and  completely ignore the fluid flow. More recently an Allen--Cahn Navier--Stokes model
for reactive one-phase flow with precipitation and dissolution was proposed in Ref.~\cite{carina19pre}.

\section{The Sharp Interface Formulation}
\label{ChapterSI}
In this section we present the free boundary problem which is the basis for the phase field approach 
that will be introduced in Section \ref{ChapterDI}. While most of the governing equations and coupling conditions resemble standard choices, we introduce 
a novel ansatz for the  momentum  in the solid phase and for  its coupling to the fluid phases. In Section \ref{sSlip} we show that this approach 
realizes a Navier-slip boundary condition for the fluid-solid interface. 

We introduce a domain $\Omega\subset \RR^N$, $N\in \set{2,3}$, and assume that it is the disjoint union of domains $\Omega_1(t)$, $\Omega_2(t)$ and $\Omega_3(t)$ for all times $t\in [0,T]$. We interpret $\Omega_1(t)$, $\Omega_2(t)$, $\Omega_3(t)$ as bulk domains which are occupied by fluid phase 1 (e.g.~water), fluid phase 2 (e.g.~oil) and a solid phase, respectively. All bulk domains are time-dependent, as the fluid bulk domains can change by convection and the solid bulk domain by precipitation and dissolution processes. As displayed in Figure \ref{Figure_SIDomain} we denote the interface between $\Omega_i$ and $\Omega_j$ by $\Gamma_{ij}$ ($i<j$). The normal unit vector $\nv\in \SSa^{N-1}$ of the interface $\Gamma_{ij}$ is supposed to point into $\Omega_j$. We call $\Gamma_{12}$ the fluid-fluid interface, and $\Gamma_{13}$ and $\Gamma_{23}$ fluid-solid interfaces. By $\nu \in \RR$ we denote the normal velocity of the interface $\Gamma_{ij}$.

\begin{figure}[!bp]
\centering
 \begin{tikzpicture}
\begin{scope}[scale = 3.5]

\draw (-1,-1) rectangle (1,.9);
\draw (0,0) to[out=125,in=-90] 
		node [fill=white,pos=0.5] {$\Gamma_{12}$} 
		node (n0) [pos=0.8,inner sep=0] {}
		node (n1) [sloped,pos=0.8,yshift=2em,inner sep=0]{}
		(-0.35,0.9);
\draw (0,0) to[out=-130,in=90] 
		node [fill=white,pos=0.5] {$\Gamma_{13}$} 
		node (n2) [pos=0.8,inner sep=0] {}
		node (n3) [sloped,pos=0.8,yshift=2em,inner sep=0]{}
		(-0.2,-1);
\draw (0,0) to[out=20,in=180] 
		node [fill=white,pos=0.5] {$\Gamma_{23}$} 
		node (n4) [pos=0.8,inner sep=0] {}
		node (n5) [sloped,pos=0.8,yshift=-2em,inner sep=0]{}
		(1,0);

\draw[-stealth] (n0) -- node[above]{$\nv$} (n1);
\draw[-stealth] (n2) -- node[above]{$\nv$} (n3);
\draw[-stealth] (n4) -- node[right]{$\nv$} (n5);

\node[align=center] at (-0.6,-0.1) {$\Omega_1$\\ Fluid phase};
\node[align=center] at (0.6,0.6) {$\Omega_2$\\ Fluid phase};
\node[align=center] at (0.6,-0.5) {$\Omega_3$\\ Solid phase};

\end{scope}
\end{tikzpicture}
%
%
 \caption{Partition of $\Omega$ into bulk domains $\Omega_1$, $\Omega_2$, $\Omega_3$ and interfaces $\Gamma_{12}$, $\Gamma_{13}$, $\Gamma_{23}$.}
  \label{Figure_SIDomain}
\end{figure}
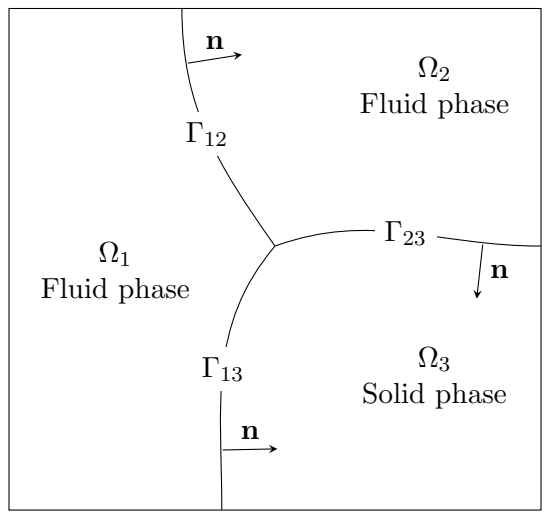

\subsection{The Bulk Equations}
We consider the incompressible flow of the viscous fluid phase $i$ in $\Omega_i(t)$, $i\in\set{1,2}$.  Then,  for a  velocity field $\vv = \vv(t,\xv)\in\RR^N$ and pressure $p=p(t,\xv)\in\RR$ the dynamics is  governed by the incompressible Navier--Stokes equations, that is
\begin{align}
\label{eq_ns1}
\nabla \cdot \vv &= 0, \\
\label{eq_ns2}
\partial_t (\rho_i \vv) + \nabla \cdot \left( \rho_i \vv \otimes \vv \right) + \nabla p &= \nabla \cdot (2 \gamma_i \nabla^s \vv)
\end{align}
in $\Omega_i(t)$, $i\in\set{1,2}$, $t\in(0,T)$. Here, the fluid density $\rho_i>0$ and the viscosity $\gamma_i>0$ are assumed to be constant but are allowed to be different for each fluid phase. The symmetric Jacobian $\nabla^s \vv$ is given by $\nabla^s \vv = \frac{1}{2}\left(\nabla \vv + (\nabla \vv)^t\right)$.

Furthermore, we assume the presence of ions that can dissolve in fluid phase 1 but not in fluid phase 2. Thus we account for the ion-concentration 
$c=c(t,\xv)\ge 0$ in $\Omega_1(t)$ which is supposed to satisfy a standard transport-diffusion equation
\begin{align}
\label{eq_c}
\partial_t c + \nabla \cdot (\vv c) - D \Delta c = 0
\end{align}
in $\Omega_1$, $t\in(0,T)$, using a constant diffusion rate $D>0$. In the solid phase we assume to have a constant ion-concentration $c^\ast>0$.

Albeit the solid phase should be  immobile we impose an artificial velocity field  $\vv=\vv(t,\xv)$ for it that is assumed to satisfy the elliptic law
\begin{align}
\label{eq_solid_v}
\nabla \cdot (2 \gamma_3 \nabla^s \vv) - \rho_3 d_0 \vv = 0
\end{align}
in $\Omega_3$, $t\in(0,T)$, with constants $\gamma_3, d_0 > 0$ and density of the solid phase $\rho_3>0$. Notably equation \eqref{eq_solid_v} has no physical meaning, but will be essential to establish a slip condition for the tangential fluid velocity at the fluid-solid interfaces $\Gamma_{13}(t)$ and $\Gamma_{23}(t)$.

\begin{figure}[!bp]
\centering
\begin{tikzpicture}
\begin{scope}[scale = 3.3]

\draw (-1,-1) rectangle (1,.9);
\draw (0,0) to[out=125,in=-90] 
		node [fill=white,pos=0.65,align=center] {$\nabla c \cdot \nv = 0$} 
		(-0.35,0.9);
\draw (0,0) to[out=-130,in=90] 
		node [fill=white,pos=0.7,align=center] {$\nu = -r(c) + \alpha \sigma_{13}\kappa$ \\  $D \nabla c \cdot \nv = \nu (c^\ast - c)$}
		(-0.2,-1);
\draw (0,0) to[out=20,in=180] 
		node [fill=white,pos=0.55,align=center] {$\nu=0$} 
		(1,0);

\node[align=center] at (-0.55,-0.05) {$\Omega_1$ \\ Transport of ions};
\node[align=center] at (0.6,0.7) { $\Omega_2$};
\node[align=center] at (0.7,-0.5) {$\Omega_3$};
\end{scope}
\end{tikzpicture}  \begin{tikzpicture}
\begin{scope}[scale = 3.3]

\draw (-1,-1) rectangle (1,.9);
\draw (0,0) to[out=125,in=-90] 
		node [fill=white,pos=0.55,align=center] {$\nu=\vv\cdot \nv$ \\$\jump{(p I - \gamma \nabla^s \vv) \cdot \nv} = \sigma_{12} \kappa \nv$} 
		(-0.35,0.9);
\draw (0,0) to[out=-130,in=90] 
		node [fill=white,pos=0.7,align=center] {$\vv \cdot \nv = 0$ \\ $\tfrac{1}{2} \nu \rho \vv \cdot \tau = \jump{\gamma \partial_\nv (\vv \cdot \tau)}$}
		(-0.2,-1);
\draw (0,0) to[out=20,in=180] 
		node [fill=white,pos=0.55,align=center] {$\vv \cdot \nv = 0$\\ $\jump{\gamma \partial_\nv (\vv \cdot \tau)} = 0$} 
		(1,0);

\node[align=center] at (-0.55,-0.05) {$\Omega_1$\\ Navier--Stokes};
\node[align=center] at (0.6,0.7) { $\Omega_2$\\ Navier--Stokes};
\node[align=center] at (0.67,-0.5) {$\Omega_3$\\ Elliptic law \\ for $\vv$};
\end{scope}
\end{tikzpicture}
%
%
 \caption{The bulk equations and interface conditions of the sharp interface model. Left: Equations for ion transport and surface reaction. Right: Flow equations and interface conditions, omitting the condition $\jump{\vv}=0$ that is valid at all interfaces.
 }
  \label{Figure_SI}
\end{figure}
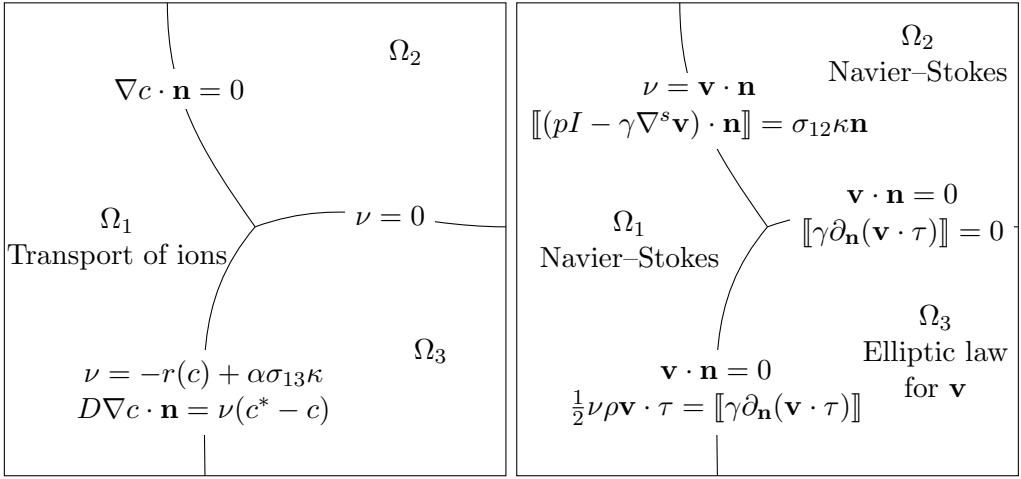

\subsection{The Interface Conditions}
\label{sInterface}
We proceed describing the interfacial dynamics  between the bulk domains. The velocity field $\vv:\Omega_1(t) \cup \Omega_2(t) \cup \Omega_3(t) \to \RR^N$ 
is assumed to be continuous across all domains, i.e.
\begin{align}
\label{eq_v_continuous}
\jump{\vv}&=0 \quad \text{ on }\Gamma_{12},\, \Gamma_{13}\text{ and }\Gamma_{23} .
\end{align}
Here $\jump{a}$ is the jump of a quantity $a(\xv)$ across an interface $\Gamma_{ij}$, that is 
\begin{align*}
\jump{a}(\xv) = \lim_{\xi \to 0^+} \left(a(\xv+\xi\nv) - a(\xv-\xi\nv) \right) \quad \text{ for } \xv\in\Gamma_{ij}.
\end{align*}

The interface conditions between two fluids are determined by the balance laws for mass and momentum. They are given for the Navier--Stokes system by
\begin{align}
\label{eq_s_v}
\nu = \vv \cdot \nv  &\quad \text{ on }\Gamma_{12}, \\
\label{eq_jump_stress}
\jump{(p I - 2 \gamma \nabla^s \vv) \cdot \nv} = \sigma_{12} \kappa \nv  &\quad \text{ on }\Gamma_{12},
\end{align}
involving the normal velocity $\nu$ of the interface, the mean curvature $\kappa$ and the (constant) surface tension coefficient $\sigma_{12}>0$ between the two fluids.

For the fluid-solid interfaces $\Gamma_{13}$ and $\Gamma_{23}$ we impose the conditions
\begin{align}
\label{eq_solid_fluid_v}
\vv \cdot \nv &= 0 &&\text{ on }\Gamma_{13}\text{ and }\Gamma_{23},   \\
\label{eq_solid_fluid_v2}
\frac{1}{2} \nu \rho \vv \cdot \tauv &= \jump{\gamma \partial_\nv (\vv \cdot \tauv)} \qquad \text{for all } \tauv \in \RR^N, \tauv \perp \nv && \text{ on }\Gamma_{13}\text{ and }\Gamma_{23}.
\end{align}
Condition \eqref{eq_solid_fluid_v} is the usual no-penetration condition for fluid flow. Condition \eqref{eq_solid_fluid_v2} will give, together with \eqref{eq_solid_v}, a slip condition for the tangential flow, see Section \ref{sSlip} for details.

\begin{remark}
\label{Remark_rho}
Instead of \eqref{eq_solid_fluid_v} one can impose the more general mass conservation $-\vv \rho_1 \cdot \nv = \nu \jump{\rho}$ on $\Gamma_{13}$. This allows for a volume change related to the reaction process. Under the assumption that the solid phase has the same density as fluid phase $1$, that is $\rho_3 = \rho_1$, there is no volume change and we recover \eqref{eq_solid_fluid_v}. For the sake of simplicity we will present the technically less involved computations resulting from \eqref{eq_solid_fluid_v}.
\end{remark}

It remains to fix the normal velocity of the fluid-solid interfaces $\Gamma_{13}$ and $\Gamma_{23}$, which is given by the rates of precipitation and dissolution. We assume that reactions can only take place between fluid phase 1 and the solid phase, excluding reactions across $\Gamma_{23}$. Precisely, we choose
\begin{align}
\label{eq_s_r_alpha}
\nu = - r(c) + \alpha \sigma_{13} \kappa &\quad \text{ on }\Gamma_{13},  \\
\label{eq_s_0}
\nu = 0 &\quad \text{ on }\Gamma_{23}.
\end{align}

The reaction rate function $r=r(c)$ depends only on the ion concentration $c$ in fluid 1 and models both, dissolution and precipitation. We follow Knabner et al.\cite{Knabner96} without introducing additional effects such as surface charge, and assume $r(c)$ to be monotonically increasing in $c$. 

\begin{remark}
A simple choice for a reaction rate $r(c)$ is given by modelling the rate of precipitation using a quadratic mass action law and the rate of dissolution using a constant rate. With reaction rates $k_1, k_2 > 0$ this means
\begin{align*}
r(c) = k_1 c^2 - k_2.
\end{align*}
\end{remark}

The term $\alpha \sigma_{13} \kappa$ in \eqref{eq_s_r_alpha} models curvature effects acting on the precipitation and dissolution process. While in previous works\cite{carina19pre,magnus} the sharp interface limit of the phase field models required a positive  $\alpha$, we also allow for $\alpha =0$ in our analysis. We will need to distinguish between the cases with and without curvature effects, that is $\alpha > 0$ and $\alpha = 0$, for the free energy functional in Section \ref{sEnergy} and for the asymptotic analysis in Section \ref{sInnerLeading}.

Finally, we need a transmission condition that ensures the conservation of ions. Recall that we assume a constant ion concentration $c^\ast$ in the solid bulk domain $\Omega_3$. For $\Gamma_{13}$ and $\Gamma_{12}$ we thus impose the Rankine-Hugoniot like conditions
\begin{align}
\label{eq_rh}
D \nabla c \cdot \nv = \nu (c^\ast - c)&\quad \text{ on }\Gamma_{13}, \\
\label{eq_rh2}
\nabla c \cdot \nv = 0&\quad \text{ on }\Gamma_{12}.
\end{align}
\subsection{The Contact Angle Condition}
The set of points where the three bulk domains $\Omega_1(t)$, $\Omega_2(t)$, $\Omega_3(t)$ meet consists of manifolds $\Gamma_{123}$ with codimension $2$.  In the two-dimensional case the domains meet at distinct points, while in the three-dimensional case they meet at lines. Let us consider the two-dimensional case first.

Given the surface tension coefficients $\sigma_{12}$, $\sigma_{13}$, $\sigma_{23} > 0$ we impose the contact angle condition
\begin{align}
\label{eq_contactangle}
\frac{\sin(\beta_1)}{\sigma_{23}} = \frac{\sin(\beta_2)}{\sigma_{13}} = \frac{\sin(\beta_3)}{\sigma_{12}},
\end{align}
at $\Gamma_{123}$, where $\beta_i$ is the contact angle of $\Omega_i$ at the contact point. Note that the $\beta_i$ are uniquely determined through \eqref{eq_contactangle} and $\beta_1 + \beta_2 + \beta_3 = 2\pi$.

In the three-dimensional case, we impose condition \eqref{eq_contactangle} on the plane perpendicular to $\Gamma_{123}$.

With this, the description of the sharp interface formulation is complete. It consists of the bulk equations \eqref{eq_ns1}-\eqref{eq_solid_v}, the interface conditions \eqref{eq_v_continuous}-\eqref{eq_rh2} and the contact angle condition \eqref{eq_contactangle}. It is necessary for the well-posedness of the sharp interface formulation that $\beta_3 = \pi$ and that we have the interface condition \eqref{eq_solid_fluid_v2} instead of a no-slip condition. Classical results\cite{Solonnikov82} show that prescribing both, a non-moving contact point and a contact angle, leads to an ill-posed model.

Additional boundary conditions have to be imposed on $\partial \Omega$, for example a Navier-slip condition for $\vv$ and a homogeneous Neumann boundary condition for $c$. For the sake of brevity we will not consider expansions close to the boundary $\partial \Omega$ in the sharp interface limit in Section \ref{ChapterSIL}. 

\subsection{The Navier-Slip Condition}
\label{sSlip}
Before we conclude this section on the sharp interface model, we investigate the effect of the bulk equation \eqref{eq_solid_v} for $\vv$ in the solid domain $\Omega_3$ together with the boundary conditions \eqref{eq_solid_fluid_v2} at the boundary of $\Omega_3$. Given a slip length $L>0$, the Navier-slip condition reads
\begin{align}
\label{eqSlipLength}
\vv \cdot \tauv = -L \partial_\nv (\vv \cdot \tauv) \qquad \text{for all } \tauv \in \RR^N, \,\tauv \perp \nv
\end{align}
at the interfaces $\Gamma_{13}$ and $\Gamma_{23}$, where all variables are evaluated from the side of the fluid bulk phase.  We will show that classical solutions to the sharp interface formulation \eqref{eq_ns1}-\eqref{eq_contactangle} approximately satisfy \eqref{eqSlipLength} with
\begin{align}
\label{eqSlipLength2}
 L = \gamma_1 \sqrt{\frac{2}{ \rho_3 d_0 \gamma_3}}.
\end{align}

For the sake of simplicity we consider a simple planar geometry, i.e. \\$\Omega_3 = \set{\xv\in\RR^N, \xv_1 < 0}$, $\Omega_1 = \RR^N \setminus \Omega_3$, $\Omega_2 = \emptyset$ and let all unknowns only depend on $\xv_1$. Then \eqref{eq_solid_v} reads as
\begin{align*}
2 \gamma_3 \partial_{\xv_1}^2 \vv(\xv_1) -  \rho_3 d_0 \vv(\xv_1) = 0.
\end{align*}
Assuming a bounded velocity profile for $\xv_1 \to -\infty$ we find
\begin{align*}
\vv = \Cv e^{\sqrt{\frac{ \rho_3 d_0}{2 \gamma_3}} \xv_1},
\end{align*}
with some vector constant $\Cv\in\RR^N$. In the solid bulk domain $\Omega_3$ we find up to the boundary $\xv_1 \to 0^-$
\begin{align}
\label{eq_slip_v2}
\partial_\nv (\vv \cdot \tauv) = -\partial_{\xv_1} (\vv \cdot \tauv) = -\sqrt{\frac{ \rho_3 d_0}{2 \gamma_3}} \vv \cdot \tauv.
\end{align}
Assume that there is no reaction, so that \eqref{eq_solid_fluid_v2} reduces to
\begin{align}
\label{eq_solid_v2}
\jump{\gamma \partial_\nv (\vv \cdot \tauv)} = 0.
\end{align}
Recall that by \eqref{eq_v_continuous} $\vv$ is continuous across the interface $\Gamma_{13}$. Therefore, with \eqref{eq_slip_v2} and \eqref{eq_solid_v2} we find at the boundary of $\Omega_1$, that is for $\xv_1\to 0^+$, the Navier-slip condition \eqref{eqSlipLength}, \eqref{eqSlipLength2}.

In a more general geometry we also expect this behavior, as long as the exponential decay of $\vv$ in the interior of $\Omega_3$ is sufficiently fast. For this, the quotient $d_0/\gamma_3$ should be large. As both, $d_0>0$ and $\gamma_3>0$, are non-physical parameters, the slip length $L$ can be chosen while keeping a large quotient $d_0/\gamma_3$.

On the left hand side of \eqref{eq_solid_fluid_v2} we have the term $\frac{1}{2} \nu \rho \vv \cdot \tauv$. This term appears in the sharp interface limit in Section \ref{sInnerFirst}. In general, we expect the normal velocity $\nu$ of a fluid-solid interface to be small, so this term has minor influence on the slip length.

\begin{remark}
To realize a no-slip condition one can choose a large $d_0$ in \eqref{eq_solid_v}. Recalling \eqref{eqSlipLength}, this leads to the slip length $L$ approaching zero. At the same time the quotient $d_0/\gamma_3$ becomes large and we have $\vv \approx 0$ in the solid domain. 

A different approach considered in Ref.~\cite{carina19pre} is to choose $\gamma_3 = 0$. Our analysis in this section does not hold in that case. Instead, \eqref{eq_solid_v} directly results in $\vv=0$ in the solid domain $\Omega_3$ and the continuity of $\vv$ in \eqref{eq_v_continuous} implies a no-slip condition for the fluid. When considering the sharp interface limit for this approach, we do not get the tangential stress balance \eqref{eq_solid_fluid_v2}, as it would over-determine the system.
\end{remark}
\section{The Phase Field Model}
\label{ChapterDI}
\subsection{Preliminaries}

To establish  a phase field model in our case  we  introduce the fields 
\begin{align*}
\phi_1(t,\xv), \phi_2(t,\xv), \phi_3(t,\xv) : [0,T]\times\Omega \to \RR
\end{align*}
that approximate the indicator function of the respective phase in the sharp interface model. We summarize the fields in the vector-valued function
$\Phiv = (\phi_1, \phi_2, \phi_3)^t$ and 
 call $\Phiv = \ev_i$ a pure phase, with $\ev_i\in\RR^3$ being the $i$-th unit vector. In contrast to the sharp interface formulation, $\phi_i$ runs smoothly between $0$ and $1$ in a small layer around the interface. The width of this diffuse transition zone is controlled by a new parameter $\eps>0$. In the limit $\eps \to 0$ the layer collapses to the interface and we expect to regain the sharp interface formulation \eqref{eq_ns1}-\eqref{eq_contactangle}. For this we will consider the sharp interface limit by asymptotic expansions in $\eps$ in Section \ref{ChapterSIL}.

Understanding the smooth phase field parameter $\phi_i$ as a volume fraction of the $i$-th phase we want to ensure that $\Phiv$  satisfies for all $t\in[0,T]$  and $\xv \in \Omega$ the conservation
 constraint
\begin{align}
\label{eq_sum_phi}
\sum_{i=1}^3 \phi_i(t,\xv) = 1, 
\end{align}
and additionally the range restriction  $\phi_i(t,\xv) \in [0,1]$. However, the phase field dynamics will rely on  the fourth--order Cahn--Hilliard evolution, which 
does not satisfy a priori such a maximum principle.  We will  enforce the relaxed constraint $\phi_i(t,\xv) \in (-\delta,1+\delta)$ for some 
small $\delta >0$ by using an unbounded  potential function. To do so, we  define first a double-well potential $W_\text{dw}(\phi)$ by
\begin{align}
\label{modelWdiph}
W_{\text{dw}}(\phi) = 18 \phi^2 (1-\phi)^2 + \delta \ell\left(\frac{\phi}{\delta} \right) + \delta \ell \left( \frac{1-\phi}{\delta} \right) ,\quad \ell(x) = \begin{cases} \frac{x^2}{1+x} & x \in (-1,0), \\ 0 & x \geq 0, \end{cases}
\end{align}
see also Figure \ref{Figure_W}.
\begin{figure}[!bp]
\centering
 \begin{tikzpicture}[samples=100]
\begin{scope}[scale = 3.6]

\draw[-Stealth] (0,0)--(1.8,0) node[below]{$\phi_i$};
\draw[-Stealth] (0,0)--(0,1.0);

\def\dta{0.4}
\def\hscale{4}
\begin{scope}[xscale = 1, shift={(\dta/1.5,0)}]

\begin{scope}[yscale = \hscale]
\draw[domain=0:1] plot (\x, {(\x)^2*(1-\x)^2}) ;
\draw[domain=-\dta/2.7:0] plot (\x, {(\x/\dta)^2/(1+\x/\dta)});
\draw[domain=1:1+\dta/2.7] plot (\x, {((1-\x)/\dta)^2/(1+(1-\x)/\dta)}) node[left]{$W_{\text{dw}}(\phi_i)$};
\end{scope}
\draw (0,0) -- (0,-1pt) node[below]{0};
\draw (1,0) -- (1,-1pt) node[below]{1};
\draw (-\dta/1.5,0) -- (-\dta/1.5,-1pt) node[below]{$-\delta$};
\draw[dotted] (1+\dta/1.5,0) -- (1+\dta/1.5,1.0);
\draw (1+\dta/1.5,0) -- (1+\dta/1.5,-1pt) node[below]{$1+\delta$};
\end{scope}
\end{scope}
\end{tikzpicture}\quad \includegraphics[height=4.6cm, trim = 75 50 75 20, clip]{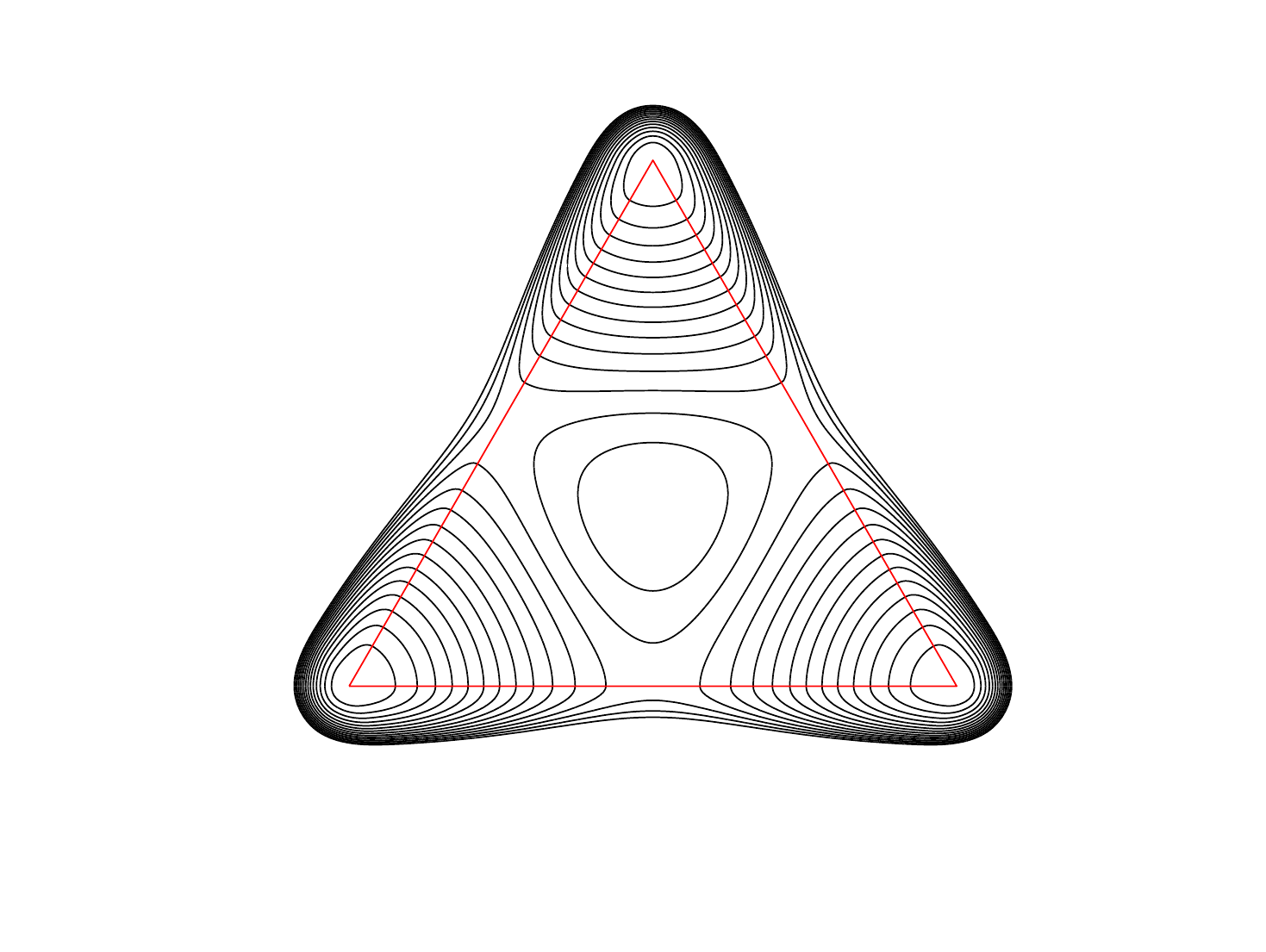}
 \caption{Left: Plot of the double well potential $W_{\text{dw}}$. Right: Contour plot of $W_0$ in barycentric coordinates.}
  \label{Figure_W}
\end{figure}
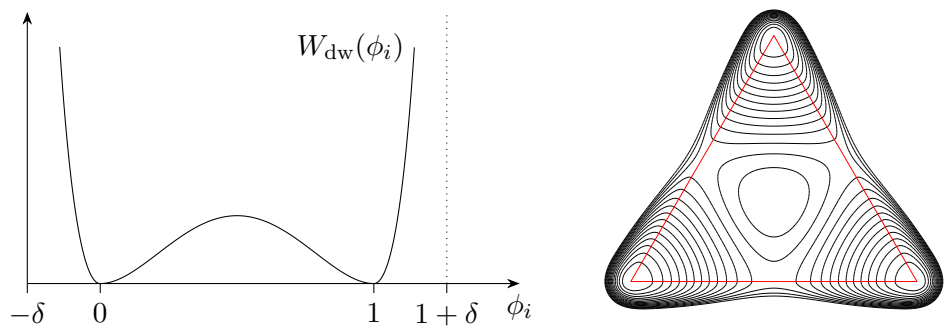
\begin{remark}[General properties of the double-well potential] 
The choice in  \eqref{modelWdiph} for $W_{\text{dw}}$ is a specific one fitting the phase field model exactly to the sharp interface model, see the analysis in 
Section \ref{ChapterSIL}.  In general   $W_\text{dw}$ should be symmetric, i.e. $W_{\text{dw}}(\phi) = W_{\text{dw}}(1-\phi)$,  it should have minima at $\phi=0$ (and $\phi=1$) and satisfy $\lim_{\phi \to -\delta^+} W_{\text{dw}}(\phi) = \infty$. 
\end{remark}

To define now the potential function  $W(\Phiv): \RR^3 \to \RR$   we note that its choice based on the double-well function  $W_{\text{dw}}$ induces  different surface energies for each 
of the interfaces by different scalings (see also Remark \ref{Remark_sigma}). Based on the work of Boyer et al.\cite{Boyer06,Boyer10} we consider
\begin{align}
\label{modelW0}
W_0(\Phiv) = \sum_{i=1}^3 \Sigma_i W_{\text{dw}}(\phi_i),
\end{align}
with scaling coefficients $\Sigma_i>0$ $(i=1,2,3)$, see also Figure \ref{Figure_W}.
Because $W_0(\Phiv)$ is only a reasonable choice for states $\Phiv$ from  the plane $\sum_i \phi_i = 1$ we introduce a projection $P$ of $\RR^3$ onto this plane by
\begin{align}
\label{eqProjection}
P\Phiv = \Phiv + \Sigma_T(1-\phi_1-\phi_2-\phi_3) \begin{pmatrix}
\Sigma_1^{-1}\\
\Sigma_2^{-1}\\
\Sigma_3^{-1}
\end{pmatrix}
,\qquad \frac{1}{\Sigma_T} = \frac{1}{\Sigma_1}+\frac{1}{\Sigma_2}+\frac{1}{\Sigma_3}.
\end{align}
With the projection we  finally define the potential $W(\Phiv):=W_0(P\Phiv)$. Note that  $W(\Phiv): \RR^3 \to \RR$ is a function  with a minimum in each of the pure phases $\Phiv=\ev_i$.
Moreover, the  choice will ensure in particular that $\Phiv$ satisfies the constraint \eqref{eq_sum_phi}. An equivalent formulation by introducing a Lagrange multiplier for the constraint is given in Ref.~\cite{Boyer06}.

\begin{remark}[Relation of $\Sigma_i$ and $\sigma_{ij}$]
\label{Remark_sigma}
Consider the two--phase case  satisfying $\phi_i+\phi_j = 1$ for $i,j\in \{1,2,3\}, i\neq j$. In this case there are only transition zones  between the pure phases $\Phiv=\ev_i$ and $\Phiv=\ev_j$. Then $W$ reduces  to the scaled double-well potential:
\begin{align*}
W(\Phiv) = \Sigma_i W_{\text{dw}}(\phi_i) + \Sigma_j W_{\text{dw}}(\phi_j) = (\Sigma_i+\Sigma_j) W_{\text{dw}} (\phi_i).
\end{align*}
 In the asymptotic analysis in Section \ref{sInnerLeading} the scaling factor $\Sigma_i + \Sigma_j$ will be identified as the surface energy
  $\sigma_{ij}$ of the sharp interface formulation \eqref{eq_ns1}-\eqref{eq_contactangle}. We therefore have
$\sigma_{ij} = \Sigma_i + \Sigma_j,\,i,j \in\set{1,2,3}$, which
leads to
\begin{align*}
\Sigma_i = \frac{1}{2}\left(\sigma_{ij} + \sigma_{ik} - \sigma_{jk}\right),\quad i,j,k \in\set{ 1,2,3},\quad i\neq j \neq k \neq i.
\end{align*}
In the literature, see e.g. Ref.~\cite{spreading}, $-\Sigma_i$ is known as wetting or spreading coefficient. A negative value of $\Sigma_i$ implies $\sigma_{jk} >\sigma_{ij} + \sigma_{ik}$, that is an interface of phases $j$ and $k$ is energetically less favorable than a thin film of phase $i$ in between these phases, phase $i$ is "spreading".

While $\sigma_{23}$ will have less impact on our model due to scaling, the surface energy $\sigma_{12}$ induces surface tension effects between the two fluids and $\sigma_{13}$ impacts the precipitation and dissolution process.
\end{remark}

\subsection{The $2f1s$-Model}
We proceed to present the complete phase field model coupling  the Cahn--Hilliard equations with the    Navier--Stokes system, describing two fluid phases plus one solid phase ($2f1s$). The total fluid fraction $\phi_f$
and the   ion--dissolving fluid fraction $\phi_c$ are given by 
\begin{align}
\label{modelPhif}
\phi_f(\Phiv) := \phi_1 + \phi_2, \qquad \phi_c := \phi_1.
\end{align}
Furthermore, we define the total density and the fluid density by
\begin{align}
\label{modelRhof}
\rho(\Phiv) := \rho_1 \phi_1 + \rho_2 \phi_2 + \rho_3 \phi_3,\qquad \rho_f(\Phiv) := \rho_1 \phi_1 + \rho_2 \phi_2.
\end{align}
To govern the three-phase dynamics we introduce  for $\eps >0$ the $2f1s$-model
\begin{align}
\nabla \cdot (\phi_f \vv) &= 0, \label{model1}\\
\begin{split}
 \partial_t (\rho_f \vv) + \nabla \cdot ((\rho_f \vv + \Jv_f) \otimes \vv) &= - \phi_f \nabla p + \nabla \cdot(2 \gamma(\Phiv) \nabla^s \vv) \\
&\qquad - \rho_3 d(\phi_f,\eps) \vv + \Sv + \frac 12 \rho_1 \vv R_f, \label{model2}
\end{split}\\
 \partial_t (\phi_c c) + \nabla \cdot ((\phi_c \vv + \Jv_c) c) &= D \nabla \cdot ( \phi_c  \nabla c) + R_c,\label{model3}\\
\partial_t \phi_i + \nabla \cdot (\phi_i \vv + \Jv_i) &= R_i, &&\hspace{-3em} i\in\set{1,2}, \label{model4}\\
\partial_t \phi_3 + \nabla \cdot \Jv_3 &= R_3,  \label{model4b}\\
\mu_i &= \frac{\partial_{\phi_i}W(\Phiv)}{\eps} -  \eps\Sigma_i \Delta \phi_i, &&\hspace{-3em}  i\in\set{1,2,3} \label{model5}
\end{align}
in $(0,T)\times\Omega$. The flux terms are given by
\begin{align}
\label{modelJ}
\Jv_i &= - \frac{\eps}{\Sigma_i} \nabla \mu_i, \quad \text{and}\quad \Jv_f = \rho_1 \Jv_1 + \rho_2 \Jv_2, \qquad \Jv_c = \Jv_1.
\end{align}
The reaction terms $R_1$, $R_2$, $R_3$, $R_c$, $R_f$ modelling  precipitation and dissolution of ions satisfy
\begin{align}
\label{modelR}
R_3 = - R_1, \qquad R_2 = 0, \quad \text{and}\quad \qquad \frac{1}{c^\ast}R_c=-R_3, \qquad R_f&= R_1
\end{align}
It remains to fix $R_1$ which will be derived  in  Section \ref{sEnergy} as a constitutive relation
from thermodynamical considerations.

The term $\Sv$ models the effective surface tension between the two fluids. There are a multitude of choices even 
for the two-phase case, see Ref.~\cite{kimST} for an overview. As generalization to the  three-phase case
which assures thermodynamical consistency (see Theorem \ref{thrmThermo})  we use
\begin{align}
\label{modelS}
\Sv &= -\mu_2 \phi_f\nabla\left(\frac{\phi_1}{\phi_f}\right) -\mu_1 \phi_f\nabla\left(\frac{\phi_2}{\phi_f}\right).
\end{align}
The $2f1s$-model \eqref{model1}-\eqref{model5} is complemented by initial conditions and is subject to the boundary conditions
\begin{align}
\label{bcv}
\vv &= 0,\\
\label{bcc}
\nabla c \cdot \nv_\Omega &= 0,\\
\label{bcphi}
\nabla \phi_i \cdot \nv_\Omega &= 0, \qquad i\in\set{1,2,3},\\
\label{bcmu}
\nabla \mu_i \cdot \nv_\Omega &= 0, \qquad i\in\set{1,2,3}
\end{align}
on $(0,T)\times \partial \Omega$. Here $\nv_\Omega\in \SSa^{N-1}$ denotes the outer normal unit vector on $\partial \Omega$.

\subsection{Discussion of the $2f1s$-Model}

\paragraph{Discussion of Equation \eqref{model1}:} Requiring  $\phi_f \vv$ to be  divergence free replaces 
the usual incompressibility constraint on $\vv$ alone. We follow here the idea of volume averaging presented by Abels et al.\cite{agg}, instead of the 
classical approach by Lowengrub\&Truskinovsky \cite{lowengrub}.  The latter would  not lead to a  divergence-free formulation which 
we favor for numerical reasons. Note that   $\vv$ in \eqref{model1}  has then to be understood as the velocity of the fluid fraction
instead of some average velocity of the full mixture.  In particular, the ansatz prevents advection of  the  phase parameter  $\phi_3$ of the
 solid phase  in the governing equation \eqref{model4b}. 

\begin{remark}  We assume like in Section \ref{ChapterSI} that  the densities $\rho_1$ and $ \rho_3$ equal.
	Otherwise, 	 the reaction process would lead to a change in volume, see also Remark \ref{Remark_rho} and we would loose the incompressibility
	constraint \eqref{model1}. 
	Note  that the  relation $R_3 = -R_1$ in \eqref{modelR} is a special case of the more general mass conservation relation
	 $R_1 \rho_1 + R_3 \rho_3 = 0$ accounting for change in volume. Equation \eqref{model1} would read in this  case as
\begin{align*}
\nabla \cdot (\phi_f \vv) = R_1 + R_3.
\end{align*}
\end{remark}

\paragraph{Discussion of Equations \eqref{model2}:} The momentum equations are formulated for the combined momentum $\rho_f(\Phiv) \vv$ of the two fluid phases and involve the pressure-like term $\phi_f \nabla p$. Note that this term is not in divergence form anymore, due to the fact that the solid phase is assumed to be immobile and can thus act as a sink or source for momentum. This becomes clear by rewriting
\begin{align*}
\phi_f \nabla p = \nabla (\phi_f p) - p \nabla \phi_f.
\end{align*}
The first term on the right hand side is now in divergence form. The second term contributes in the interfacial region between the solid and the fluid phases, with $\nabla \phi_f$ being orthogonal to the interface here. It is therefore a normal force acting between the solid phase and the fluid phases.

The viscosity $\gamma$ in \eqref{model2} depends on the phase field parameter $\Phiv$. We choose harmonic averaging of the bulk viscosities from Section \ref{ChapterSI}, i.e.
\begin{align}
\label{model_gamma}
\gamma(\Phiv) = \left( \phi_1\gamma_1^{-1} +  \phi_2\gamma_2^{-1} +  \phi_3\gamma_3^{-1} \right)^{-1}.
\end{align}
Whereas $\gamma_1$ and $\gamma_2$ are physical quantities, note that $\gamma_3$ does not represent a physical viscosity and is used for the Navier-slip condition instead.

In Ref.~\cite{agg} a thermodynamically consistent Cahn--Hilliard model for two-phase flow is constructed by adding a flux term in the momentum equations. We generalize this approach to an additional solid phase by the term $\nabla \cdot (\Jv_f \otimes \vv)$ and obtain a thermodynamically consistent model, see Theorem \ref{thrmThermo} below. The phase field parameter gets transported by both, the fluid fraction velocity $\vv$ and the Cahn--Hilliard fluxes $\Jv_i$. This leads to an additional transport of the momentum of each fluid phase with its respective flux $\Jv_i$.

Next, we discuss the term $d(\phi_f,\eps) \vv$. Here $d(\cdot , \eps)$ can be any smooth, decreasing function with $d(1,\eps)=0$, $d(0,\eps)=d_0>0$ for a constant $d_0$ independent of $\eps$. This term ensures that $\vv$ is small in the solid phase. Similar ideas have been used in Ref.~\cite{Beckermann,carina19pre,ShapeOptim}. While these works get $\vv=0$ in the solid phase, we use the variable to allow for slip at the fluid-solid interface instead, see Section \ref{sSlip}.\\[10pt]
\textbf{Discussion of Equation \eqref{model3}:} The equation for the dissolved ion concentration $c$ consists of transport, diffusion and reaction. Analogously to the momentum equations, we introduce the additional term $\nabla \cdot (\Jv_c c)$ to account for the transport caused by the Cahn--Hilliard equation. The rate of diffusion scales with $\phi_c$, such that there is no diffusion through other phases.\\[10pt]
\textbf{Discussion of Equations \eqref{model4}, \eqref{model4b}, \eqref{model5}:} The phase field parameters $\phi_i$ are governed by a Cahn--Hilliard evolution. It is well known that one can interpret this evolution as a gradient flow to the Ginzburg--Landau free energy
\begin{align}
\label{eq_ginzburg}
f(\Phiv, \nabla \Phiv) &= \frac{W(\Phiv)}{\eps} + \sum_{i=1}^3 \frac {\eps \Sigma_i }{2} |\nabla \phi_i|^2. 
\end{align}
Pure phases $\Phiv = \ev_i$ are minima of the potential $W(\Phiv)$. Phase transitions, that are characterized by large gradients, are penalized in \eqref{eq_ginzburg} through the term $|\nabla \phi_i|^2$. These two energy contributions get weighted by the parameter $\eps$, resulting in phase transitions with a width of order $\eps$. Following Boyer et al.\cite{Boyer06,Boyer10} the coefficients $\Sigma_i$ have no influence on the width of the diffuse transition zone.

The Cahn--Hilliard equations \eqref{model4}, \eqref{model4b}, \eqref{model5} are coupled to the Navier--Stokes equations \eqref{model1}, \eqref{model2} through the advection of $\phi_1$ and $\phi_2$. The solid phase $\phi_3$ is not advected, leading to an effective total flow velocity of $\phi_f \vv$, as described above.

As we will see in Section \ref{sConserved}, solutions to our model satisfy $\sum_{i=1}^3 \phi_i = 1$ and $\sum_{i=1}^3 \mu_i = 0$. As a consequence one of the  equations for the three phase field parameters can be eliminated.

\subsection{The $\delta$-$2f1s$-Model}

For the $2f1s$-model we are not able to achieve thermodynamical consistency without the following modification.
We need to avoid that quantities like $\phi_f$ and $\rho_f$ from \eqref{modelPhif} and \eqref{modelRhof} can attain negative values, leading to a 
 degeneration of the model. Therefore, we redefine these quantities using the small parameter $\delta$  which has been used 
 to define the double-well potential in \eqref{modelWdiph}.
\begin{align}
&\tilde \phi_f := \phi_1+\phi_2 + 2\delta \phi_3, \label{rmodel_phif}\\
&\tilde \rho_f(\Phiv) := \rho_1 \phi_1 + \rho_2 \phi_2 + (\rho_1+\rho_2) \delta, \label{rmodel_rhof}\\ 
&\tilde \gamma(\Phiv) := \left(\phi_1 \gamma_1^{-1}+ \phi_2 \gamma_2^{-1} + \phi_3 \gamma_3^{-1}  + (\gamma_1^{-1}+\gamma_2^{-1}+\gamma_3^{-1})\delta\right)^{-1}, \label{rmodel_gamma}\\
&\tilde \phi_c := \phi_1 + \delta. \label{rmodel_phic}
\end{align}

It is straightforward to see that these quantities  are positive if  $\phi_i > -\delta$ and \eqref{eq_sum_phi} hold.
Note that the  double-well function $W_\text{dw}(\phi_i)$ from \eqref{modelWdiph} diverges at $\phi_i = -\delta$ and $\phi_i = 1+\delta$. This will imply $\phi_i \in (-\delta,1+\delta)$ by establishing an energy estimate in Section \ref{sEnergy}.

 We proceed to formulate the $\delta$-$2f1s$-model by
\begin{align}
\nabla \cdot (\tilde \phi_f \vv) &= 0, \label{rmodel1}\\
\begin{split}
 \partial_t (\tilde\rho_f \vv) + \nabla \cdot ((\rho_f \vv + \Jv_f) \otimes \vv) &= - \tilde\phi_f \nabla p + \nabla \cdot (2 \tilde \gamma(\Phiv) \nabla^s \vv) \\
 &\qquad - \rho_3 d(\tilde\phi_f,\eps) \vv + \tilde \Sv + \frac 12 \rho_1 \vv R_f, \label{rmodel2}
 \end{split}\\
 \partial_t (\tilde\phi_c c) + \nabla \cdot ((\phi_c \vv + \Jv_c) c) &= D \nabla \cdot ( \tilde\phi_c  \nabla c) + R_c,\label{rmodel3}\\
\partial_t \phi_i + \nabla \cdot (\phi_i \vv + \Jv_i) &= R_i, &&\hspace{-3em}  i\in\set{1,2}, \label{rmodel4}\\
\partial_t \phi_3 + \nabla \cdot (2\delta \phi_3 \vv + \Jv_3) &= R_3,  \label{rmodel4b}\\
\mu_i &= \frac{\partial_{\phi_i}W(\Phiv)}{\eps} -  \eps\Sigma_i \Delta \phi_i, &&\hspace{-3em}  i\in\set{1,2,3} \label{rmodel5}
\end{align}
in $(0,T)\times\Omega$. The modification also affects the surface tension term $\Sv$, such that we are led to replace $\Sv$ in \eqref{model2} by $\tilde \Sv$ with
\begin{align}
\label{rmodelS}
\tilde \Sv &= -\mu_2 \tilde \phi_f\nabla\left(\frac{\phi_1}{\tilde\phi_f}\right) -\mu_1 \tilde\phi_f\nabla\left(\frac{\phi_2}{\tilde\phi_f}\right)  - 2\delta \phi_3 \nabla (\mu_3-\mu_1-\mu_2) .
\end{align}

\begin{remark} Note that for $\delta \to 0$ the double-well function $W_{\text{dw}}(\phi_i)$ converges point-wise to a potential of double-obstacle type, i.e.
\begin{align}
W_{\text{dw},0}(\phi_i) = \phi_i^2 (1-\phi_i)^2 + \ell_o\left(\phi_i\right) + \ell_o\left( 1-\phi_i\right) ,\quad \ell_o(x) = \begin{cases} \infty &  x < 0 \\ 0 & x \geq 0 \end{cases}
\end{align}
and we need to interpret $W_{\text{dw},0}'$ as a set-valued subderivative. The Cahn--Hilliard equation with double obstacle-potential has been thoroughly studied, see for example Ref.~\cite{bloweyElliott}. 
While this ansatz does not require any modification to $\phi_f$, $\rho_f$, $\phi_c$ and $\gamma$, the resulting model will include variational inequalities, which we aim to avoid.
\end{remark}

From here on we will consider only the $\delta$-$2f1s$-model \eqref{rmodel1}-\eqref{rmodel5}.
\subsection{Numerical Example}
\label{sNumerics}
Before we analyze the $\delta$-$2f1s$-model, let us illustrate the capability of the model by a numerical example. The equations were discretized using the Galerkin-FEM method. Taylor-Hood elements were used for $\vv$ and $p$, and $P_2$-Lagrange elements for $\phi_1$, $\phi_2$, $\mu_1$, $\mu_2$, $c$. The implementation was done in PDELab\cite{pdelab} using DUNE-ALUGrid\cite{alugrid}.

We consider  initially a  solid nucleus ($\phi_3$, red) in a channel flow ($\phi_1$, dark blue). 
Attached is a part of the second fluid phase ($\phi_2$, light blue).  The initial datum is displayed in  Figure \ref{Figure_Nucleus}, top left.
The upper and lower boundaries are impermeable while the left(right) boundary  acts as inflow(outflow) boundary.
Due to a flow from the left, the second fluid phase gets pushed behind the nucleus (see Figure \ref{Figure_Nucleus}, top right/bottom left.
Because the ion concentration at  the inflow boundary is oversaturated, the nucleus begins to grow as can clearly be seen from the last graph in 
Figure \ref{Figure_Nucleus}.

\begin{figure}[!bh]
\begin{center}
\includegraphics[width=0.4\textwidth]{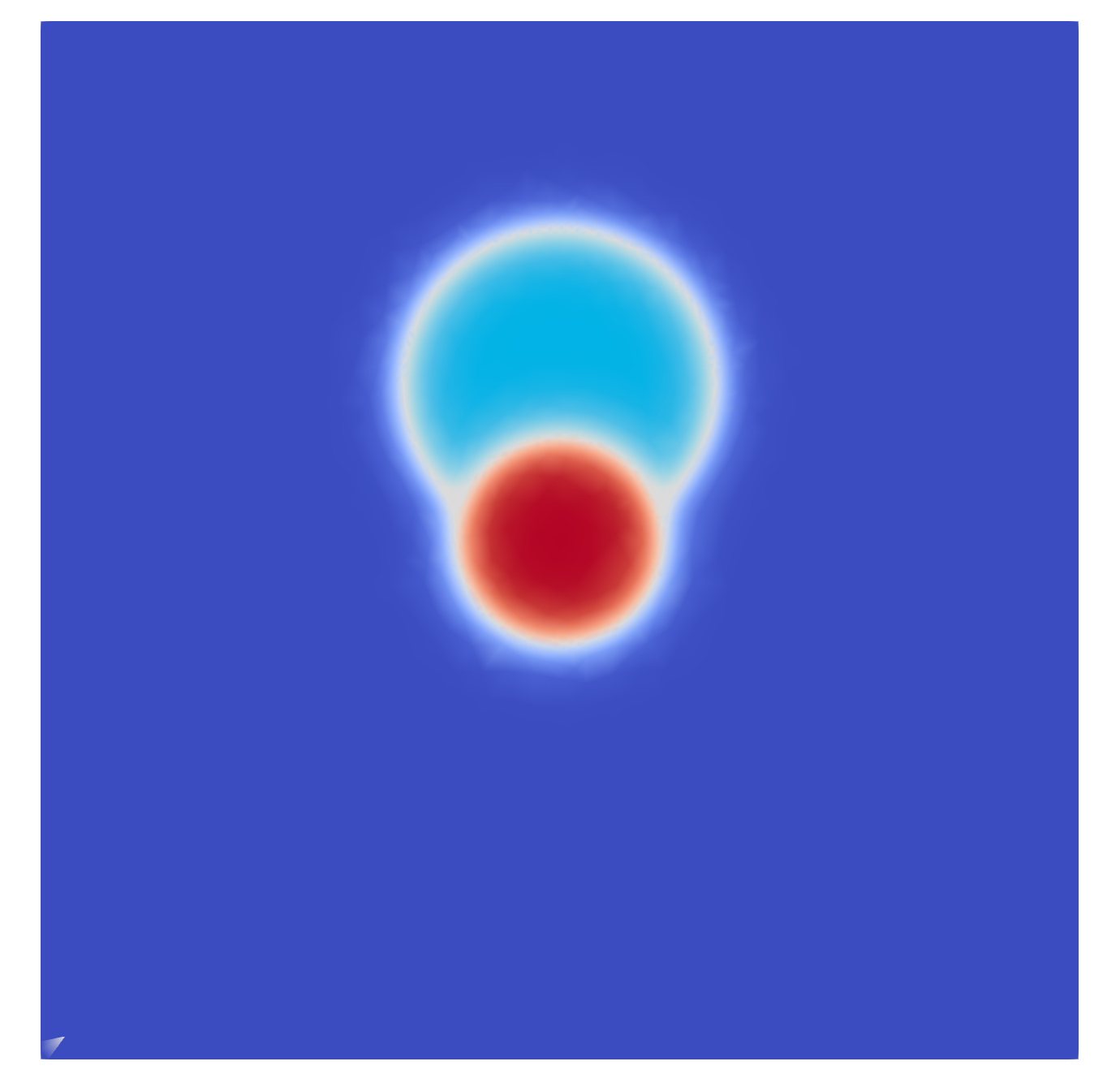}
\includegraphics[width=0.4\textwidth]{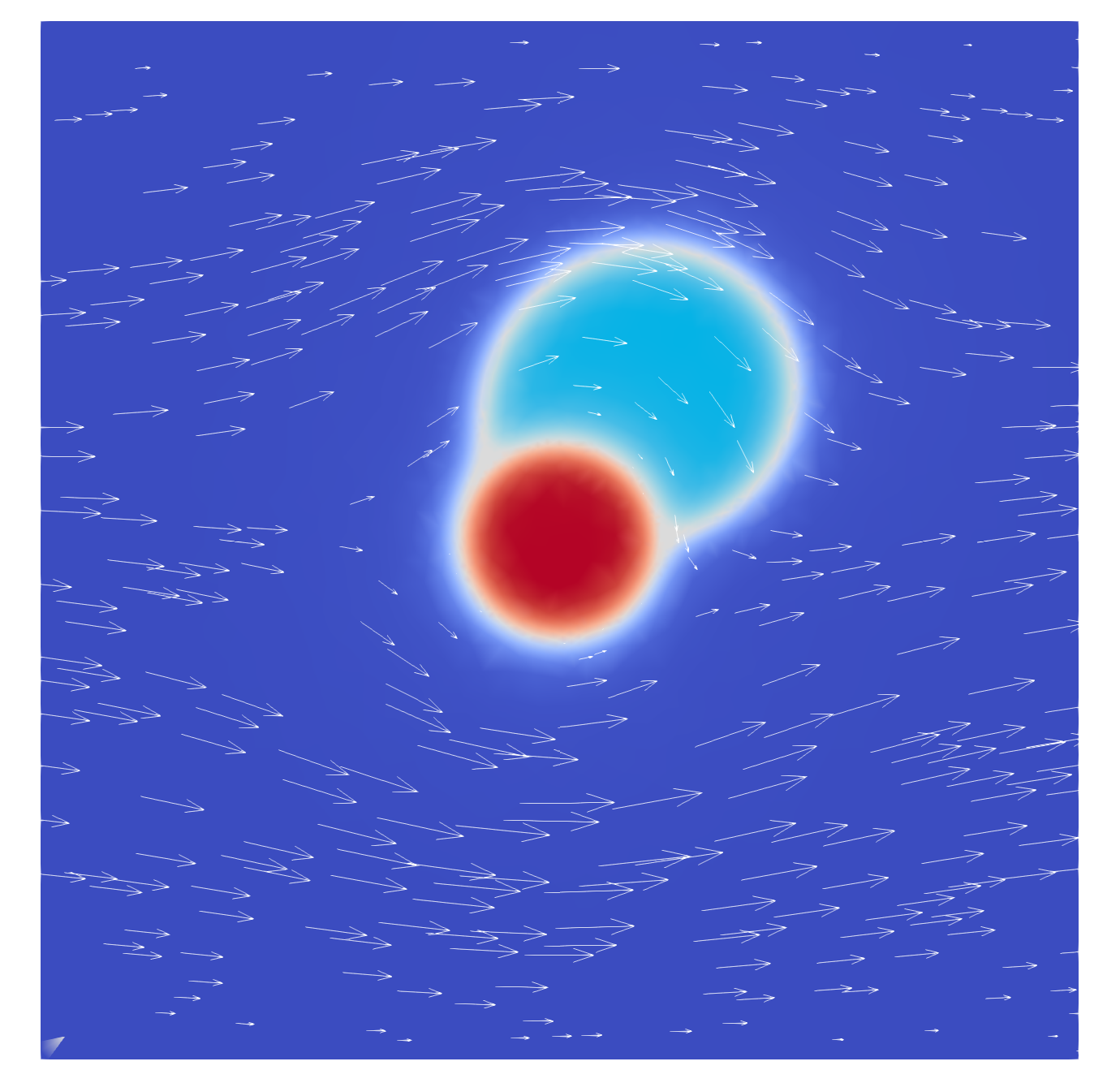}
\\
\includegraphics[width=0.4\textwidth]{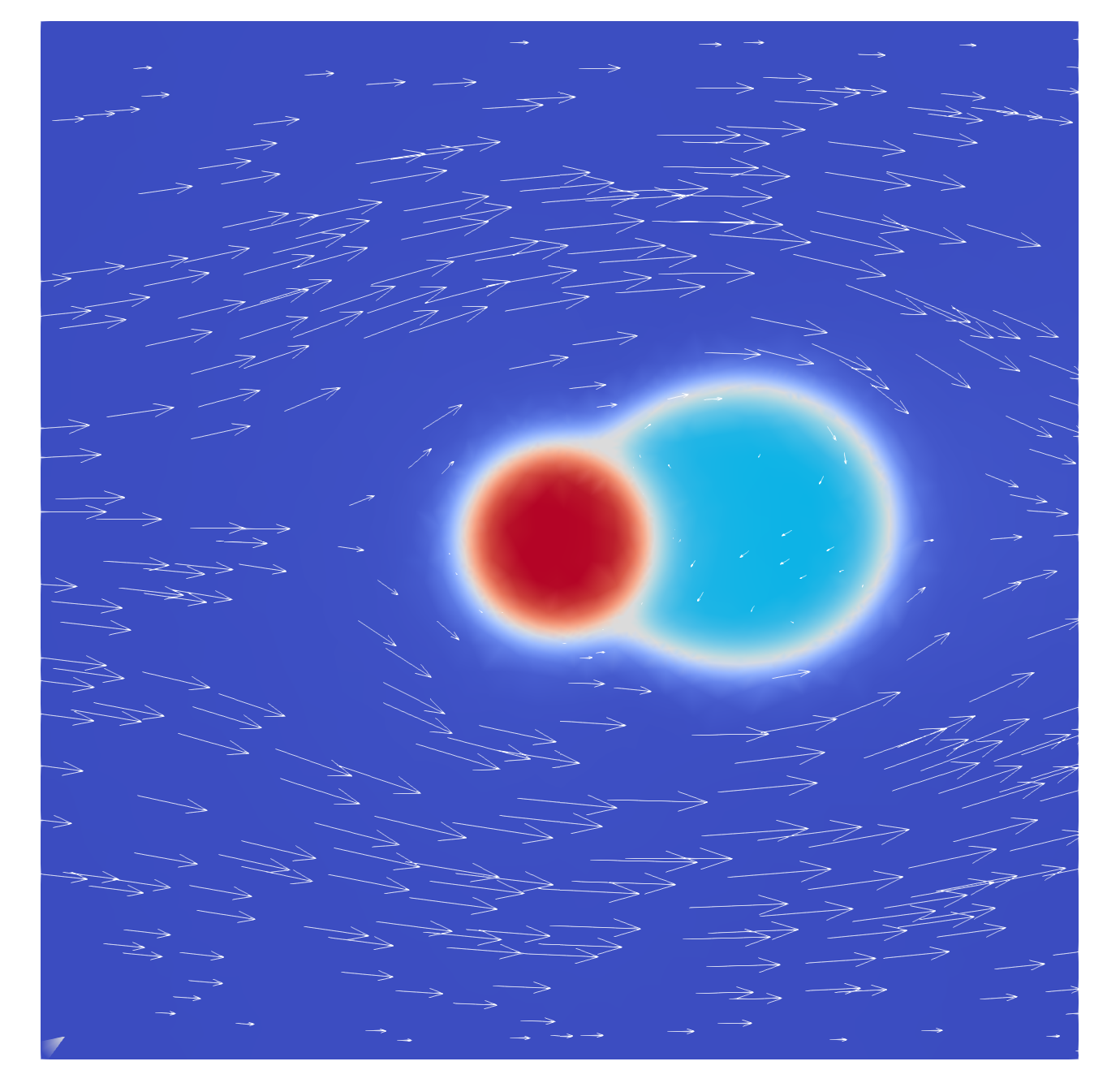}
\includegraphics[width=0.4\textwidth]{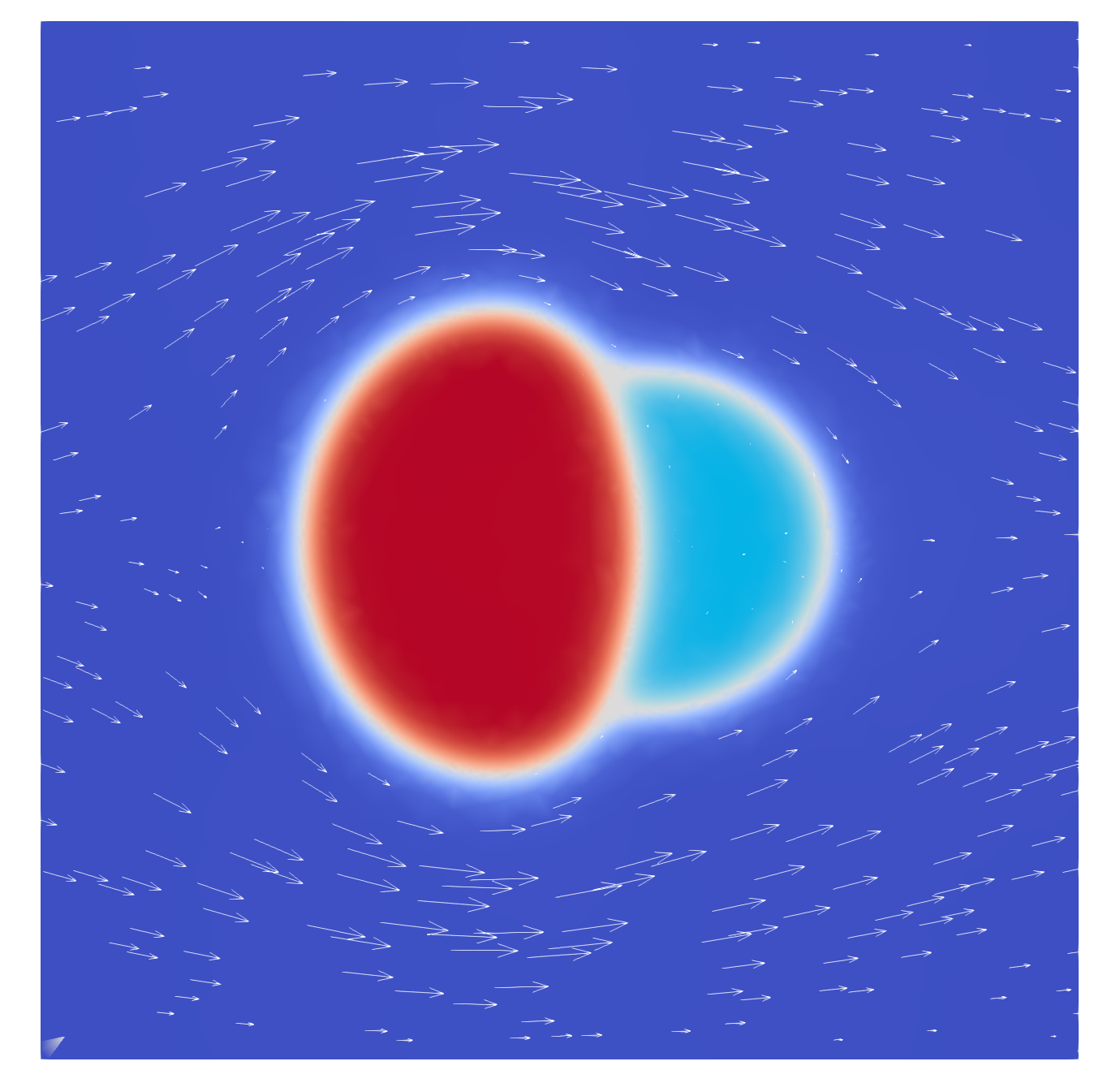}
\end{center}
\caption{Growth of a solid nucleus in a channel flow, with attached fluid phase $\Phiv=\ev_2$. Top left: initial data, top right: $t=0.5$, bottom left: $t=2$, bottom right: $t=8$}
\label{Figure_Nucleus}
\end{figure}

\subsection{Conservation of Total Mass, Ions and Volume Fraction}
\label{sConserved}
Consider the $\delta$-$2f1s$-model with boundary conditions \eqref{bcv}-\eqref{bcmu}. The phase field equations are written in divergence form, so it is easy to see that for classical solutions we have
\begin{align*}
\frac{\diff}{\diff t} \int_\Omega \phi_i \diff \xv &= \int_\Omega R_i \diff \xv,
\end{align*}
i.e. the phase field variables are balanced by the reaction terms only. Using \eqref{modelR} this implies that the total mass $\rho(\Phiv)$ from \eqref{modelRhof} is conserved, that is
\begin{align*}
\frac{\diff}{\diff t} \int_\Omega \rho(\Phiv) \diff \xv&= \int_\Omega \rho_1 R_1 + \rho_2 R_2 + \rho_1 R_3 \diff \xv = 0.
\end{align*}
Moreover, the total amount of ions, given by $\tilde \phi_c c + \phi_3 c^\ast$, is invariant because \eqref{modelR} implies
\begin{align*}
\frac{\diff}{\diff t} \int_\Omega \tilde \phi_c c + \phi_3 c^\ast \diff \xv &= \int_\Omega R_c + c^\ast R_3 \diff \xv = 0.
\end{align*}

Finally, classical solutions of the $\delta$-$2f1s$-model obey also \eqref{eq_sum_phi} provided \eqref{eq_sum_phi} is satisfied initially. This is due to our construction of the triple-well function $W(\Phiv)=W_0(P\Phiv)$ with the projection $P$ from in \eqref{eqProjection}. $W$ is constant in the direction $(\Sigma_1^{-1},\Sigma_2^{-1},\Sigma_3^{-1})^t$, therefore
\begin{align*}
\sum_{i=1}^3 \frac{\mu_i}{\Sigma_i} &= \frac{1}{\eps}\sum_{i=1}^3 \frac{ \partial_{\phi_i} W(\Phiv)}{\Sigma_i} - \eps \Delta \sum_{i=1}^3 \phi_i = \frac{\partial_{(\Sigma_1^{-1},\Sigma_2^{-1},\Sigma_3^{-1})^t} W(\Phiv)}{\eps} - \eps \Delta 1 = 0,
\end{align*}
and thus we get the desired conservation of volume fractions as
\begin{align*}
\frac{\diff} {\diff t}\sum_{i=1}^3 \phi_i &= \sum_{i=1}^3 R_i - \nabla \cdot (\tilde \phi_f \vv)+ \eps \Delta \sum_{i=1}^3 \frac{\mu_i}{\Sigma_i} = 0.
\end{align*}

\subsection{Thermodynamical Consistency} 
\label{sEnergy}

Interpreting the Cahn--Hilliard equation as a gradient flow of the Ginzburg--Landau energy \eqref{eq_ginzburg} and following the ideas in Ref.~\cite{agg} the $\delta$-$2f1s$-model can be shown to be thermodynamically consistent. That is, we find a free energy functional satisfying a dissipation inequality along the evolution of the $\delta$-$2f1s$-model. In our case it is
\begin{align}
\label{eqEnergy}
F(\Phiv,\nabla\Phiv,v,c) &= \int_\Omega \frac 12 \tilde \rho_f |\vv|^2 + f(\Phiv, \nabla \Phiv) + \frac{1}{\tilde \alpha} g(c) \tilde \phi_c \diff \xv.
\end{align}
This free energy functional consists of three parts: The kinetic energy of the fluid phases, the Ginzburg--Landau energy \eqref{eq_ginzburg}, and a third term $\tilde \alpha^{-1} g(c) \tilde \phi_c$. The last term represents the free energy associated with the fluid-ions mixture, depending only on the ion concentration. Note that precipitation and dissolution can increase the surface area between the fluid and the solid phase (and thus the Ginzburg--Landau energy $f$), so they have to decrease the free mixture energy $g(c)$ at the same time. 

With this in mind, we choose the specific form of the up to now free function $R_1$ as
\begin{align}
\label{modelrR}
r(c) := g(c) + g'(c)(c^\ast - c), \qquad R_1  = -\frac{q(\Phiv)}{\eps}  \left(r(c)+ \tilde \alpha \mu_1-\tilde \alpha \mu_3\right).
\end{align}
The choice of $r(c)$ is motivated by the following. Consider \eqref{rmodel3}-\eqref{rmodel5} for constant initial values and with $\vv=0$. From equations \eqref{rmodel3}, \eqref{rmodel4} we can infer the conservation of ions $\partial_t(\tilde\phi_c c + (1-\tilde\phi_c) c^\ast ) = 0$, and have therefore an implicit relation for $c = c(\tilde \phi_c)$. Under these conditions $r$ is given by the derivative of the free energy with respect to $\tilde \phi_c$, that is
\begin{align*}
r &= \frac{\diff }{\diff \tilde \phi_c }\Big(g(c(\tilde \phi_c))\tilde \phi_c\Big) .
\end{align*}
As stated in Section \ref{sInterface} we consider reaction terms $r(c)$ that are increasing in $c$.   A short calculation shows that there is in fact a bijection between convex $g(c)$ and increasing $r(c)$. We will therefore assume $g(c)$ to be convex in the following.

The reaction term $R_1$ does not only depend on $r$ but also on the phase field potentials $\mu_1$ and $\mu_3$. These represent the influence of curvature effects on the reaction. As described in \eqref{eq_s_r_alpha} this effect should scale with a chosen constant $\alpha \in [0,\infty)$. The case $\alpha = 0$ requires extra care. We therefore introduce a modified $\alpha$ as
\begin{align}
\label{rmodel_alpha}
\tilde \alpha = \begin{cases}
\alpha & \text{for } \alpha > 0,\\
\eps & \text{for } \alpha = 0.
\end{cases}
\end{align}

Furthermore, we localize the reaction to the fluid-solid interface by choosing the function $q(\Phiv)$ as
\begin{align*}
q(\Phiv) = 6 \phi_1\phi_3.
\end{align*}
\begin{remark}
\label{Remark_q}
This is a similar choice as in Ref.~\cite{magnus}, where fluid motion is ignored. The situation is more intricate here. In general, we require $q=0$ when $\phi_1=0$ or $\phi_3 = 0$. Furthermore, across an interface between phase $\Phiv=\ev_1$ and phase $\Phiv=\ev_3$ we require $q(\Phiv) = \sqrt{2 W_{\text{dw}}(\phi_1)}$.
\end{remark}
To derive a thermodynamically consistent model we had to introduce the flux terms in \eqref{modelJ} and the specific choice of the reaction $R_1$ in \eqref{modelrR}, and can now prove the following theorem.
\begin{theorem}
\label{thrmThermo}
Classical solutions to the $\delta$-$2f1s$-model which obey the boundary conditions \eqref{bcv}-\eqref{bcmu} and satisfy $F(\Phiv,\nabla\Phiv,v,c) < \infty$ initially
 fulfill for all $t\in (0,T]$ the free energy
 dissipation inequality
\begin{align*}
\frac{\diff }{\diff t} F(\Phiv,\nabla\Phiv,v,c)  &= \int_\Omega - 2 \tilde \gamma(\Phiv) \nabla \vv\dbdot \nabla^s \vv -\rho_3 d(\tilde \phi_f,\eps) \vv^2 - \eps \sum_{i=1}^3 \frac{1}{\Sigma_i}|\nabla \mu_i|^2 \\ 
&\quad - \tilde \alpha^{-1} D g''(c) \tilde \phi_c |\nabla c|^2 - \frac{q(\Phiv)}{\eps \tilde \alpha} \left(r(c)+\tilde \alpha \mu_1-\tilde \alpha \mu_3\right)^2 \diff \xv \\
&\leq 0 .
\end{align*}
\end{theorem}

\begin{proof}
We treat the time derivative of each of the terms in \eqref{eqEnergy} separately. Let us start with $\partial_t(\tilde \phi_c g(c) )$. Using integration by parts and the homogeneous boundary conditions we have
\begin{align}
\begin{split}
\label{eqThermoG1}
g'(c) \nabla \cdot ( (\phi_c \vv + \Jv_c) c) &= \int_\Omega (\nabla g(c)) \cdot  (\phi_c \vv + \Jv_c) + g'(c) c \nabla \cdot (\phi_c \vv + \Jv_c) \diff \xv\\
&= \int_\Omega -(g(c) - g'(c)c)  \nabla \cdot (\phi_c \vv + \Jv_c)\diff \xv.
\end{split}
\end{align}
Using \eqref{rmodel3}, \eqref{rmodel4} and \eqref{eqThermoG1} we calculate
\begin{align}
\begin{split}
\label{eqThermoG2}
\int_\Omega \partial_t(g(c) \tilde \phi_c) \diff \xv &= \int_\Omega g'(c) \partial_t(\tilde \phi_c c) + (g(c) -g'(c) c)\partial_t \tilde \phi_c \diff \xv \\
&= \int_\Omega g'(c) \big(R_{c} - \nabla \cdot ((\phi_c \vv+ \Jv_c)c) +  D\nabla \cdot (\tilde \phi_c \nabla c)\big) \\
&\qquad + (g(c) -g'(c) c) \big(R_1-\nabla \cdot (\phi_1 \vv+\Jv_1)\big) \diff \xv \\
&= \int_\Omega g'(c) D\nabla \cdot (\tilde \phi_c \nabla c) + R_1(g(c) -g'(c) c + g'(c) c^\ast) \diff \xv \\
&= \int_\Omega g'(c) D\nabla \cdot (\tilde \phi_c \nabla c) + r(c) R_1 \diff \xv.
\end{split}
\end{align}
We require some results from vector calculus: For vector fields $\Av, \Bv$ we have
\begin{align*}
\nabla \cdot (\Av \otimes \Bv) = (\nabla \cdot \Av) \Bv + (\Av \cdot \nabla) \Bv, \qquad
\Bv \cdot \Big( (\Av \cdot \nabla) \Bv \Big) = \frac{1}{2} (\Av \cdot \nabla) |\Bv|^2.
\end{align*}
Using this and partial integration we get
\begin{align}
\begin{split}
\label{eqThermoV1}
\int_\Omega \vv \cdot \Big(\nabla \cdot ((\rho_f \vv +\Jv_f)\otimes \vv)\Big) \diff \xv 
&= \int_\Omega  \vv \cdot \Big((\nabla \cdot (\rho_f \vv+\Jv_f)) \vv + ((\rho_f \vv+\Jv_f) \cdot \nabla) \vv\Big) \diff \xv \\
&= \int_\Omega \vv^2 \nabla \cdot (\rho_f \vv+\Jv_f) + \frac{1}{2} ((\rho_f \vv+\Jv_f) \cdot \nabla) \vv^2 \diff \xv \\
&= \int_\Omega \frac{1}{2} \vv^2 \nabla \cdot (\rho_f \vv+\Jv_f)\diff \xv.
\end{split}
\end{align}
Also, note that
\begin{align}
\begin{split}
\label{eqThermoV2}
\partial_t \tilde \rho_f &= \rho_1 \partial_t \phi_1 + \rho_2 \partial_t \phi_2 \\
&=\rho_1 (R_1-\nabla \cdot (\phi_1 \vv+ \Jv_1)) + \rho_2 (R_2-\nabla \cdot(\phi_2 \vv + \Jv_2)) \\
&=\rho_1 R_1 - \nabla \cdot (\rho_f \vv + \Jv_f).
\end{split}
\end{align}
With \eqref{eqThermoV1} and \eqref{eqThermoV2} we can calculate the time derivative of the kinetic energy as
\begin{align}
\begin{split}
\label{eqThermoV3}
\frac{\diff}{\diff t} \int_\Omega \frac 12 \tilde \rho_f |\vv|^2 \diff \xv
&= \int_\Omega  \vv \cdot \partial_t(\tilde \rho_f \vv) - \frac 12 \vv^2 \partial_t \tilde \rho_f \diff \xv
\\
&= \int_\Omega  - \vv \cdot \big(\nabla \cdot ((\rho_f \vv+\Jv_f)\otimes \vv) \big)  - \tilde \phi_f \vv \cdot \nabla p 
\\
&\qquad  + \vv \cdot \left(\nabla \cdot (2 \tilde \gamma(\Phiv) \nabla^s \vv)\right)
-\rho_3 d(\tilde\phi_f,\eps) \vv^2 +\tilde \Sv\cdot \vv +\vv\cdot \frac 12 \rho_1 \vv R_f
\\
&\qquad - \frac{1}{2}\vv^2 \rho_1 R_1 + \frac{1}{2} \vv^2 \nabla \cdot (\rho_f \vv+ \Jv_f) \diff \xv
\\
&= \int_\Omega \vv \cdot \left(\nabla \cdot (2 \tilde \gamma(\Phiv) \nabla^s \vv)\right) -\rho_3 d(\tilde\phi_f,\eps) \vv^2 +\tilde \Sv\cdot \vv \diff \xv.
\end{split}
\end{align}
Next we consider the surface tension terms. Note that by \eqref{rmodel1}
\begin{align*}
0=\nabla \cdot (\tilde \phi_f \vv) = \nabla\cdot(\phi_1 \vv)+\nabla\cdot(\phi_2 \vv)+\nabla\cdot(2\delta\phi_3 \vv),
\end{align*} 
and using this, partial integration and \eqref{rmodelS} we find
\begin{align}
\begin{split}
\label{eqThermoS1}
&\int_\Omega \mu_1 \nabla\cdot(\phi_1 \vv) + \mu_2 \nabla\cdot(\phi_2 \vv) + \mu_3 \nabla \cdot(2\delta \phi_3 \vv) \diff \xv \\
&\qquad= \int_\Omega -\mu_2 \nabla\cdot(\phi_1 \vv) -\mu_1 \nabla\cdot(\phi_2 \vv) + (\mu_3-\mu_1-\mu_2) \nabla \cdot(2\delta \phi_3 \vv) \diff \xv\\
&\qquad=\int_\Omega -\mu_2 \tilde\phi_f\nabla\left(\frac{\phi_1}{\tilde\phi_f}\right) \cdot \vv -\mu_1 \tilde\phi_f\nabla\left(\frac{\phi_2}{\tilde\phi_f}\right) \cdot \vv - 2\delta \phi_3 \nabla (\mu_3-\mu_1-\mu_2) \cdot \vv \diff \xv\\
&\qquad=\int_\Omega \tilde \Sv\cdot \vv \diff \xv .
\end{split}
\end{align}
With \eqref{eqThermoS1} we calculate the time derivative of the Ginzburg--Landau energy \eqref{eq_ginzburg} to be
\begin{align}
\begin{split}
\label{eqThermoF1}
\frac{\diff}{\diff t} \int_\Omega f(\Phiv, \nabla \Phiv) \diff \xv &= \int_\Omega \sum_{i=1}^3 \left( \frac{\partial f}{\partial \phi_i} - \nabla \cdot \frac{\partial f}{\partial \nabla \phi_i}\right) \partial_t \phi_i \diff \xv = \int_\Omega \sum_{i=1}^3 \mu_i \partial_t \phi_i \diff \xv \\
&= \int_\Omega \sum_{i=1}^3 \mu_i (R_i - \nabla \cdot \Jv_i) -\mu_1 \nabla\cdot(\phi_1 \vv) \\
&\qquad - \mu_2 \nabla\cdot(\phi_2 \vv) - \mu_3 \nabla \cdot(2\delta \phi_3 \vv) \diff \xv \\
&= \int_\Omega (\mu_1 - \mu_3) R_1 - \sum_{i=1}^3 \mu_i \nabla \cdot \Jv_i -\tilde \Sv\cdot \vv \diff \xv .
\end{split}
\end{align}
Finally we calculate with \eqref{eqThermoG2}, \eqref{eqThermoV3} and \eqref{eqThermoF1} for the complete expression
\begin{align*}
\frac{\diff }{\diff t} F(\Phiv,\nabla\Phiv,v,c)&\diff \xv = \frac{\diff}{\diff t} \int_\Omega 
		\frac 12 \tilde \rho_f |\vv|^2
		 + f(\Phiv, \nabla \Phiv)
		 + \tilde \alpha^{-1} g(c) \tilde \phi_c \diff \xv
\\
&= \int_\Omega \vv\cdot \left(\nabla \cdot (2 \tilde \gamma(\Phiv) \nabla^s \vv)\right) -\rho_3 d(\tilde\phi_f,\eps) \vv^2 +\tilde \Sv\cdot \vv
 + (\mu_1 - \mu_3) R_1 
\\
&\qquad	- \sum_{i=1}^3 \mu_i \nabla \cdot \Jv_i -\tilde \Sv\cdot \vv 
+ \tilde \alpha^{-1} D g'(c) \nabla \cdot (\tilde \phi_c \nabla c) + \tilde \alpha^{-1} r(c) R_1	 \diff \xv		
\\
&= \int_\Omega - 2 \tilde \gamma(\Phiv) \nabla \vv \dbdot \nabla^s \vv -\rho_3 d(\tilde\phi_f,\eps) \vv^2 + \sum_i \nabla \mu_i  \cdot \Jv_i \diff \xv\\
&\qquad -\tilde \alpha^{-1} D g''(c)  \tilde \phi_c |\nabla c|^2 + \tilde \alpha^{-1} \left(r(c) + \tilde \alpha \mu_1 - \tilde \alpha \mu_3\right) R_1 \diff \xv.
\end{align*}
A straightforward calculation shows $\nabla \vv \dbdot \nabla^s \vv \geq 0$. The assertion of the theorem follows by inserting the definitions of $\Jv_i$ \eqref{modelJ} and $R_1$ \eqref{modelrR}.
\end{proof}

\subsection{Algebraic Consistency}
\label{sAlgebra}
If one of the three phases is not present, we obtain simplified scenarios which reduce to phase field models that are partly known from literature.\\
 We will study  the cases  with one phase already absent initially.
  As in Ref.~\cite{Boyer06} we will first show that this phase  will not appear at a later point in time. Afterwards we investigate the reduced two-phase models that arise from this simplification. 

Let in the following $i,j,k\in\set{1,2,3}$, $i\neq j\neq k \neq i$. We consider the case $\phi_i= 0$ and $\phi_k+\phi_j = 1$. Using \eqref{modelW0} and \eqref{eqProjection} we calculate
\begin{align*}
\partial_{\phi_i} W(\Phiv) &= \partial_{\phi_i} W_0(P\Phiv) \\
&= \left(1-\frac{\Sigma_T}{\Sigma_i}\right)\partial_{\phi_i} W_0(\Phiv) - \frac{\Sigma_T}{\Sigma_j} \partial_{\phi_j} W_0(\Phiv) - \frac{\Sigma_T}{\Sigma_k} \partial_{\phi_k} W_0(\Phiv) \\
&= \left(\Sigma_i-\Sigma_T\right)\partial_{\phi_i} W_{\text{dw}}(\phi_i) - \Sigma_T \partial_{\phi_j} W_{\text{dw}}(\phi_j) - \Sigma_T \partial_{\phi_k} W_{\text{dw}}(\phi_k) \\
&= 0.
\end{align*}
For the last step recall the definition of $W_{\text{dw}}$, \eqref{modelWdiph}, to see that $\partial_{\phi_i} W_{\text{dw}}(\phi_i) = 0$. Furthermore with $\phi_k+\phi_j=1$ and the symmetry of $W_{\text{dw}}(\phi_i)$ with respect to $\phi_i=1/2$ we have
\begin{align*}
\partial_{\phi_j} W_{\text{dw}}(\phi_j) = \partial_{\phi_j} W_{\text{dw}}(1-\phi_j) = - \partial_{\phi_k} W_{\text{dw}}(\phi_k) .
\end{align*}
Now let us assume initially $\phi_i(0,\cdot) \equiv 0$. We then have
\begin{align*}
\mu_i = \frac{\partial_{\phi_i}W(\Phiv)}{\eps} -  \eps\Delta \phi_i = 0,
\end{align*}
and therefore $\Jv_i = -\eps\Sigma_i^{-1} \nabla \mu_i = 0$. It follows that
\begin{align*}
\partial_t \phi_i = R_i - \nabla \cdot \Jv_i = 0,
\end{align*}
as for $i=1$ or $i=3$ we have $q(\Phiv)=0$ and therefore $R_i$ = 0 in all cases.

This means that $\phi_i\neq 0$ will not appear spontaneously, but only if enforced e.g. through boundary conditions. For the homogeneous boundary conditions of Theorem \ref{thrmThermo} we have $\phi_i\equiv 0$ for all times. 

Before we consider special choices we point out  another simplification for two-phase flow. 
 With the two conditions $\phi_i=0$ and $\phi_j+\phi_k=1$ we can reduce the model to a model for a single phase-field variable, say $\phi_j$. Using \eqref{modelW0} and \eqref{eqProjection} we calculate 
\begin{align*}
\partial_{\phi_j} W(\Phiv) &= \partial_{\phi_j} W_0(P\Phiv) \\
&= \left(1-\frac{\Sigma_T}{\Sigma_j}\right)\partial_{\phi_j} W_0(\Phiv) - \frac{\Sigma_T}{\Sigma_k} \partial_{\phi_k} W_0(\Phiv) - \frac{\Sigma_T}{\Sigma_i} \partial_{\phi_i} W_0(\Phiv) \\
&= \left(\Sigma_j-\Sigma_T\right)\partial_{\phi_j} W_{\text{dw}}(\phi_j) - \Sigma_T \partial_{\phi_k} W_{\text{dw}}(\phi_k) - \Sigma_T \partial_{\phi_i} W_{\text{dw}}(\phi_i) \\
&= \Sigma_j\partial_{\phi_j} W_{\text{dw}}(\phi_j),
\end{align*}
and define
\begin{align*}
\mu := \frac{\mu_j}{\Sigma_j} = \frac{\partial_{\phi_j}W_{\text{dw}}(\phi_j)}{\eps} -  \eps \Delta \phi_j.
\end{align*}

\subsubsection{Solid Phase plus one Fluid Phase ($\delta$-$1f1s$)}
We consider first the two cases $i=1$ and $i=2$. That is, one of the two fluid phases is not present in the model. As a phase field variable we choose the indicator of the remaining fluid phase. That is for $i=1$ we choose $j=2$ and for $i=2$ we choose $j=1$. Note that as calculated above $\mu_i = 0$ and $\Jv_i=0$. The model $\delta$-$2f1s$-model reduces to
\begin{align}
\nabla \cdot (\tilde \phi_f \vv) &= 0, \label{sfmodel1}\\
 \partial_t (\tilde \rho_f \vv) + \nabla \cdot ((\vv - \eps \nabla \mu) \otimes \rho_j \vv) &= -\tilde \phi_f \nabla p + \nabla \cdot (2 \tilde \gamma(\Phiv) \nabla^s \vv) \\
 &\qquad -\rho_3 d(\tilde \phi_f,\eps)\vv + \tilde \Sv + \frac12 \rho_1 \vv R_f,  \label{sfmodel2}\\ 
\partial_t (\tilde \phi_c c) + \nabla \cdot ((\phi_c \vv + \Jv_c)c) &= D\nabla\cdot(\tilde \phi_c \nabla c) + R_c,\label{sfmodel3}\\
\partial_t \phi_j + \nabla \cdot (\phi_j \vv - \eps \nabla \mu) &= R_j, \label{sfmodel4}\\
\mu &= \frac{\partial_{\phi_j}W_{\text{dw}}(\phi_j)}{\eps} -  \eps\Delta \phi_j \label{sfmodel5}
\end{align}
in $(0,T)\times\Omega$. This is a $1f1s$-model for a fluid fraction $\tilde \phi_f = 2 \delta + (1-2\delta) \phi_j$. 
Previously suggested  phase field models for single phase flow with precipitation\cite{carina19pre,noorden2011} are based on the
Allen--Cahn equation and were only able to ensure a global bound on but no dissipation  of the free energy.
 By Theorem \eqref{thrmThermo} the $\delta$-$1f1s$-model \eqref{sfmodel1}-\eqref{sfmodel5}
 for two-phase flow with precipitation/dissolution is also the first phase field model that ensures  energy dissipation.
 
The effective surface tension term reduces to
\begin{align*}
\tilde \Sv = - 2 \delta \phi_3 \nabla(\mu_3-\mu_j) = 2 \delta \sigma_{j3}  (1-\phi_j) \nabla \mu,
\end{align*}
i.e. is only there to keep consistency with the modified $\tilde \phi_f$. 

In the case $i=2$ this model is fully coupled. But for $i=1$ there is no fluid present that dissolves the ions ($\phi_c=\phi_1=0$). Then $R_f = R_c = R_j = 0$ and the ion conservation law \eqref{sfmodel3} is decoupled from the other equations and equals the diffusion equation $\delta\partial_t c = \delta D \Delta c$.

\subsubsection{Two Fluid Phases ($\delta$-$2f0s$)}
We consider the case of two fluid phases. That is we have $\phi_3 = 0$ and reduce the model to the phase field variable $\phi_1$. Note that $\phi_f=\tilde \phi_f = \phi_1+\phi_2= 1$ and $\Sigma_1^{-1}\mu_1 + \Sigma_2^{-1}\mu_2 = 0$. With this the $\delta$-$2f1s$-model reduces to
\begin{align}
\nabla \cdot \vv &= 0, \label{ffmodel1}\\
 \partial_t (\tilde \rho_f \vv) + \nabla \cdot ((\rho_f \vv + \Jv_f) \otimes \vv) + \nabla p &= \nabla \cdot (2 \tilde \gamma(\Phiv) \nabla^s \vv) + \tilde \Sv,  \label{ffmodel2}\\ 
\partial_t (\tilde \phi_c c) + \nabla \cdot ((\phi_c \vv + \Jv_c) c) &= D \nabla \cdot (\tilde \phi_c \nabla c),\label{ffmodel3}\\
\partial_t \phi_1 + \vv\cdot \nabla \phi_1 &= \nabla \cdot (\eps \nabla \mu),\label{ffmodel4}\\
\mu &= \frac{\partial_{\phi_1}W_{\text{dw}}(\phi_1)}{\eps} -  \eps\Delta \phi_1.\label{ffmodel5}
\end{align}
Note that equation \eqref{ffmodel3} does not couple back to the other equations, it is just an advection-diffusion equation for the ion concentration $c$. 

Let us calculate 
\begin{align*}
\Jv_f &= \rho_1 \Jv_1 + \rho_2 \Jv_2 = -\eps (\rho_1 \nabla \mu - \rho_2 \nabla \mu) = \eps (\rho_2 - \rho_1) \nabla \mu
\end{align*}
and
\begin{align*} 
\tilde \Sv&= -\mu_2 \nabla\phi_1 -\mu_1 \nabla\phi_2 \\
&= \Sigma_2 \mu \nabla\phi_1 -\Sigma_1 \mu \nabla(1-\phi_1) \\
&= \sigma_{12} \mu \nabla \phi_1 \\
&= \sigma_{12}\left(\frac{\partial_{\phi_1}W_{\text{dw}}(\phi_1)}{\eps} -  \eps\Delta \phi_1\right) \nabla \phi_1 \\
&= \sigma_{12}\left(\frac{\nabla W_{\text{dw}}(\phi_1)}{\eps} + \eps \nabla\left( \frac{|\nabla \phi_1|^2}{2}\right) - \eps\nabla \cdot (\nabla\phi_1\otimes\nabla \phi_1)\right)
\end{align*}
We can absorb the first two terms by defining a modified pressure $\hat p$. Overall, the momentum equation can now be expressed as
\begin{align*}
 &\partial_t (\tilde \rho_f \vv) +\nabla \cdot (\rho_f \vv \otimes \vv) +\nabla \cdot ((\rho_2 -\rho_1) \eps \nabla \mu \otimes \vv) + \nabla \hat p \\
 &\qquad = \nabla \cdot (2 \tilde \gamma(\Phiv) \nabla^s \vv) - \sigma_{12} \eps\nabla \cdot (\nabla\phi_1\otimes\nabla \phi_1).
\end{align*}
The system is, except for the $\delta$-modification of $\tilde \rho_f$ and $\tilde \gamma$, the diffuse-interface model proposed by Abels, Garcke and Gr\"un\cite{agg} for two-phase flow. 

\section{The Sharp Interface Limit}
\label{ChapterSIL}
We use matched asymptotic expansions to show that the formal asymptotic limit of the $\delta$-$2f1s$-model for $\eps\to 0$ is given by the sharp interface formulation \eqref{eq_ns1}-\eqref{eq_contactangle} presented in Section \ref{ChapterSI}. 
This  technique for the sharp interface limit has been pioneered in Ref.~\cite{Caginalp88} validating the overall phase field modelling approach.\\ 
We will  first introduce the setup and assumptions of our asymptotic analysis. Then we investigate the bulk phases of the system by
 introducing outer expansions. For the interfaces between two phases we introduce inner expansions and matching conditions. In particular we will
 recover all transmission conditions between the phases as introduced in Section \ref{ChapterSI}.
 Finally we 
 consider the triple point by a special expansion.

\subsection{Assumptions and Outer Expansions}
\label{sAssumptions}
An important choice of scaling is $\delta = \eps$, so the $\delta$-modifications vanish in the sharp interface limit $\eps \to 0$. With this choice the structure of the triple-well function $W(\Phiv)$ depends on $\eps$.

We are interested in a regime of solutions where bulk phases, characterized through small gradients in the phase field parameter $\Phiv$, are separated by interfaces. In this regime we assume that $\mu_i$ is only of order $O(1)$, not of order $O(\eps^{-1})$. This can be expected on a $O(1)$-timescale, for a detailed discussion, see Ref.~\cite{pego}. In this regime we also assume that the three bulk phases meet in the two-dimensional case at distinct points, called triple points. In the three-dimensional case they meet at distinct lines, called triple lines.

We assume that we have classical solutions of the $\delta$-$2f1s$-model with finite free energy \eqref{eqEnergy}. This implies in particular $\phi_i \in (-\delta,1+\delta)$.

The minimizers of the Ginzburg--Landau free energy \eqref{eq_ginzburg} that connect $\Phiv = \ev_i$ with $\Phiv = \ev_j$ only attain values along the edge between $\ev_i$ and $\ev_j$ because we followed the construction of Boyer in Ref.~\cite{Boyer06}. As in Ref.~\cite{stinner18pre} we therefore assume that there are no third-phase contributions in the interfacial layers. See \eqref{eq_thirdphasecontribution} below for the exact formulation of this assumption.

We assume now, that away from the interface we can write solutions to the $\delta$-$2f1s$-model in terms of \textbf{outer expansions} of the unknowns $\Phiv$, $\vv$, $p$, $c$, $\mu_1$, $\mu_2$, $\mu_3$. That is, we can write them (exemplarily for $\Phiv$) in the form
\begin{align*}
\Phiv^o(t,\xv) &= \Phiv^o_0(t,\xv) + \eps \Phiv^o_1(t,\xv) + \eps^2 \Phiv^o_2(t,\xv) + \ldots\;,
\end{align*}
where $\Phiv^o_k$, $k\in \NN_0$ do not depend on $\eps$. In particular, we 
use this notation also for non-primary variables, e.g.
\begin{align*}
\tilde \phi^o_{f} &= \phi^o_{f,0} + \eps \phi^o_{f,1} + \ldots
= \left(\phi^o_{1,0}+\phi^o_{2,0}\right) + \eps \left( \phi^o_{1,1} + \phi^o_{2,1} + 2\phi^o_{3,0}\right) + \ldots\;.
\end{align*}

To group terms by powers of $\eps$, we use Taylor expansions of the nonlinearities. If the respective derivatives exist, we have for a generic function $h$ and variable $u=u_0 + \eps u_1 + \ldots$ the expansion
\begin{align*}
h(u) &= h(u_0 + \eps u_1 +\ldots) = h(u_0) + \eps h'(u_0) u_1 + O(\eps^2) .
\end{align*}

\subsection{Solution of Outer Expansions}
\label{sOuter}
\paragraph{Expansion of \eqref{rmodel5}, $O(\eps^{-1})$:}
We first note that $\phi^o_{i}(0,\cdot) \in [0,1]$, as otherwise a small $\eps$ would result in $W(\Phiv^o)=\infty$. To determine the leading order terms, we have to consider three different cases. 

Let us first look at points with $\phi^o_{i,0} \in (0,1)$ for all $i\in\set{1,2,3}$. In this case only the polynomial part of $W_{\text{dw}}$ does contribute to the equations. The leading order terms are
\begin{align}
\label{eqdWdphi}
\partial_{\phi_i} W(\Phiv^o_0) = 0 \qquad i\in\set{1,2,3}.
\end{align}
After some tedious but straightforward calculations, we can find a $\Phiv^o_0$ satisfying \eqref{eqdWdphi}. It is unstable in the sense that it is not a local minimum of $W$. 

Next, consider the case of $\phi^o_{i,0} = 0$ for exactly one $i\in\set{1,2,3}$. We therefore have the expansion $\phi_i^o = \eps \phi^o_{i,1} + O(\eps^2)$ and 
\begin{align}
\begin{split}
\label{eqdWdphi2}
W'_{\text{dw}}(\phi_i^o) &= 2 \phi_i^o - 4 (\phi_i^o)^3 + \ell'\left(\frac{\phi_i^o}{\eps}\right) -  \ell'\left(\frac{1-\phi_i^o}{\eps}\right) \\
&= \ell'(\phi^o_{i,1} + O(\eps)) + O(\eps) \\
&= \ell'(\phi^o_{i,1}) + O(\eps).
\end{split}
\end{align}
Using \eqref{eqdWdphi2} and the identity
 \begin{align*}
\partial_{\phi_i} W(\Phiv) = (\Sigma_i - \Sigma_T) W'_{\text{dw}}(\phi_i) - \Sigma_T W'_{\text{dw}}(\phi_j) - \Sigma_T W'_{\text{dw}}(\phi_k),
\end{align*}
the leading order terms of \eqref{rmodel5} for phase $i$ are given by
\begin{align*}
0=(\Sigma_i - \Sigma_T) \ell'(\phi^o_{i,1}) - \Sigma_T W'_{\text{dw}}(\phi^o_{j,0}) - \Sigma_T W'_{\text{dw}}(\phi^o_{k,0}).
\end{align*}
With $\phi^o_{j,0}+\phi^o_{k,0}=1$ we conclude $\ell'(\phi^o_{i,1})=0$. Thus we have $\phi^o_{i,1}\geq 0$. The leading order terms of $\eqref{rmodel5}$ for the phases $j$ and $k$ result in an unstable solution $\Phiv^o_0$ at $\phi_i = \phi_j = 1/2$.

The last case is $\phi^o_{i,0} = \phi^o_{j,0} = 0$, $\phi^o_{k,0}=1$. With calculations analogous to the previous case, the leading order terms of \eqref{rmodel5} for $\phi_k$ are given by
\begin{align*}
0=-(\Sigma_k - \Sigma_T) \ell'(-\phi^o_{k,1}) - \Sigma_T \ell'(\phi^o_{i,1}) - \ell'(\phi^o_{j,1}).
\end{align*}
As $\Sigma_k - \Sigma_T > 0$ and $\ell' \geq 0$, this implies $\ell'(\phi^o_{i,1}) = \ell'(\phi^o_{j,1}) = \ell'(-\phi^o_{k,1}) = 0$. Thus $\phi^o_{i,1}, \phi^o_{j,1}\geq 0$, $\phi^o_{k,1}\leq 0$. The equations resulting from leading order terms of the other two phases are then trivially fulfilled as well. We have $\Phiv^o_0 = \ev_k$, and this is a stable solution, as it is a local minimum of $W$.

Overall, the only stable solutions to the leading order terms are the pure phases $\Phiv^o_0 = \ev_k$, $k\in\set{1,2,3}$, with the restriction $\phi^o_{i,1}, \phi^o_{j,1}\geq 0$, $\phi^o_{k,1}\leq 0$. The set of points where $\Phiv^o_0 = \ev_k$ corresponds to the bulk domain $\Omega_k$ of the sharp interface formulation described in Section \ref{ChapterSI}.

\paragraph{Expansion of \eqref{rmodel1}, $O(1)$:}
Recall the definition of $\phi_f$ in \eqref{rmodel_phif}. For the case $\Phiv^o_0 = \ev_3$ we do not retrieve any equation because $\phi_f = O(\eps)$. Otherwise we get
\begin{align*}
\nabla \cdot \vv^o_0 = 0,
\end{align*}
which is equation \eqref{eq_ns1} of the sharp interface formulation.

\paragraph{Expansion of \eqref{rmodel3}, $O(1)$:}
In the case $\Phiv^o_0 = \ev_1$ we note that $\tilde \phi_c = 1 + O(\eps)$ and $q\equiv 0$ hold. With this the leading order terms are
\begin{align*}
\partial_t c^o_0 + \nabla \cdot (\vv^o_0 c^o_0) = D \Delta c^o_0 ,
\end{align*}
which is equation \eqref{eq_c} of the sharp interface formulation. In the other cases $\Phiv^o_0 = \ev_2$, $\Phiv^o_0 = \ev_3$ we do not recover any equation.

\paragraph{Expansion of \eqref{rmodel2}, $O(1)$:}
For $\Phiv^o_0 = \ev_1$ or $\Phiv^o_0 = \ev_2$ we have $\tilde \phi_f = 1 + O(\eps)$. Note also that in these cases $R_f=O(\eps)$, $\tilde \Sv=O(\eps)$ and $d(\phi_f,\eps)=O(\eps)$. We retrieve the momentum equations \eqref{eq_ns2}, e.g. for $\Phiv^o_0 = \ev_1$
\begin{align*}
\partial_t (\rho_1 \vv^o_0) + \nabla \cdot (\vv^o_0  \otimes (\rho_1 \vv^o_0)) + \nabla p^o_0 &= \nabla \cdot (2 \gamma(\Phiv^o_0) \nabla^s \vv^o_0).
\end{align*}
On the other hand, for $\Phiv^o_0=\ev_3$ the highest order terms result in
\begin{align*}
\nabla \cdot (2 \gamma(\Phiv^o_0) \nabla^s \vv^o_0) - \rho_3 d(0,\eps)\vv^o_0=0 ,
\end{align*} 
which is equation \eqref{eq_solid_v} of the sharp interface formulation.

\subsection{Inner Expansions and Matching Conditions}
As seen in Section \ref{sOuter}, there are three stable phases, namely $\Phiv^o_0 = \ev_1, \ev_2, \ev_3$. We therefore need to focus on the interfaces between these phases. To do so, we introduce
\begin{align}
\label{eq_gamma}
\Gamma_{ij}(t) = \set{\xv \in \Omega: \phi_i(t,\xv) = \phi_j(t,\xv), \phi_i(t,\xv) > 1/3}.
\end{align}
By our assumption, $\Gamma_{ij}$ is a smooth $(d-1)$-dimensional manifold embedded in $\Omega$ and depending on time. Let $\sv$ be a local parametrization of $\Gamma_{ij}$, so that $\xv(t,\sv)\in\Gamma_{ij}$. By $\nv$ we denote the normal unit vector of $\Gamma_{ij}$, pointing from phase $i$ into phase $j$ for $i<j$. We can use this to define local curvilinear coordinates $(\zeta,\sv)$ near the interface $\Gamma_{ij}$ through
\begin{align*}
\xv(t, \sv, \zeta) &= \xv(t,\sv) + \zeta \nv(t,\sv),
\end{align*}
see Figure \ref{Figure_coord} for an illustration.
\begin{figure}[!bp]
\centering
 \begin{tikzpicture}
\begin{scope}[scale = 3.5]

\draw (0,0) to[out=-30,in=160] 
		node (nx) [fill=white,pos=0.3] {$\Gamma_{ij}$} 
		node (n4) [pos=0.6,inner sep=0] {}
		node (n5) [sloped,pos=0.6,yshift=3em,inner sep=0]{}
		node (n1) [sloped, circle, fill=black,pos=0.6,inner sep=0,minimum size=4pt, yshift=4em] {}
		node (n0) [circle, fill=black,pos=0.6,inner sep=0,minimum size=4pt] {}
		(2,0);

\node[below right] at (n0) {$\xv(t,\sv)$};
\node[right] at (n1) {$\xv(t, \sv,\zeta)$};
\draw[-stealth] (n4) -- node[right]{$\nv(t,\sv)$} (n5);
\node[below right, xshift = 1em, yshift=-0.5em] at (nx) {$\Omega_i$};
\node[above right, xshift = 1em, yshift=0.5em] at (nx) {$\Omega_j$};
\end{scope}
\end{tikzpicture}
 \caption{Local curvilinear coordinates for the interface $\Gamma_{ij}(t)$}
  \label{Figure_coord}
\end{figure}
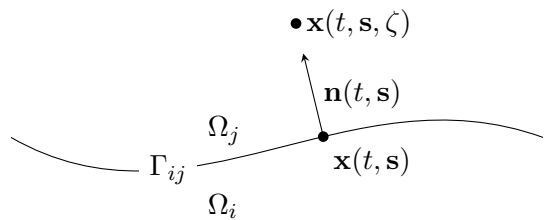
We expect the diffuse interface width being proportional to $\eps$. Therefore let $z=\zeta/\eps$ be a rescaled signed distance to the interface. We denote by
\begin{align*}
\nu = \partial_t \xv(t,\sv) \cdot \nv
\end{align*}
the normal velocity of the interface. For generic scalar and vectorial variables $u$ and $\Uv$ we obtain the transformation rules (see  Ref.~\cite{Caginalp88} and the Appendix of Ref.~\cite{agg})  
\begin{align}
\partial_t u &= - \frac 1\eps \nu \partial_z u + O(1), \\
\nabla u &= \frac 1\eps \partial_z u \nv + \nabla_\Gamma u + O(\eps), \\
\nabla \cdot \Uv &= \frac 1\eps \partial_z \Uv \cdot \nv + \nabla_\Gamma \cdot \Uv + O(\eps),\\
\nabla \Uv &= \frac 1\eps \partial_z \Uv \otimes \nv + \nabla_\Gamma \Uv + O(\eps), \\
\Delta u &= \frac 1{\eps^2} \partial_{zz} u + \frac 1\eps \kappa \partial_z u + O(1),
\end{align}
where $\kappa$ is the mean curvature of $\Gamma_{ij}$ and $\nabla_\Gamma$ denotes the surface gradient on $\Gamma_{ij}$.

We assume that close to the interface we can write solutions to the $\delta$-$2f1s$-model in terms of \textbf{inner expansions} of the form
\begin{align*}
\Phiv^{in}(t,\sv,z) &= \Phiv^{in}_0(t,\sv,z) + \eps \Phiv^{in}_1(t,\sv,z) + \ldots\, ,
\end{align*}
and similarly for all other unknowns.

For outer expansions and fixed $t$ denote the limit $\xv(t,\sv,\zeta) \to \xv(t,\sv,0)$ from positive $\zeta$ by $\sv_+$ and from negative $\zeta$ by $\sv_-$. We match the limit values of outer expansions at $\sv_+$ and $\sv_-$ with the values of the inner expansions for $z\to \pm \infty$. That is, following Ref.~\cite{Caginalp88} we impose for $\Phiv$ (and analogously for all other unknowns) the matching conditions
\begin{align}
\label{matching0}
\Phiv^{in}_0(t,\sv,\pm \infty) &= \Phiv^o_0(t,\sv_\pm), \\
\label{matchingD0}
\partial_z \Phiv^{in}_0(t,\sv,\pm \infty) &= 0, \\
\label{matchingD1}
\partial_z \Phiv^{in}_1(t,\sv,\pm \infty) &= \nabla \Phiv^o_0(t,\sv_\pm) \cdot \nv.
\end{align}
In particular, combining \eqref{matching0} and \eqref{matchingD1} we have for the velocity
\begin{align}
\label{matchingJacobian0}
\partial_z \vv^{in}_1(t,\sv,\pm \infty) \otimes \nv + \nabla_\Gamma \vv^{in}_0(t,\sv,\pm \infty) = \nabla \vv^o_0(t,\sv_\pm).
\end{align}

\subsection{Solution of Inner Expansions, Leading Order}
\label{sInnerLeading}
\paragraph{Expansion of \eqref{rmodel5}, $O(\eps^{-1})$:}
As discussed in Section \ref{sAssumptions} we assume no third-phase contributions in the interfacial layer. In detail, this means that at the interface $\Gamma_{ij}$ we assume $\phi^{in}_{k,0}=0$, where $k\neq i, k\neq j$ is the index of the third phase. We get 
\begin{align}
\label{eq_thirdphasecontribution}
\phi^{in}_{i,0} + \phi^{in}_{j,0} = 1.
\end{align}
The leading order expansion of \eqref{rmodel5} for the third phase $k$ reads
\begin{align*}
0 = (\Sigma_k - \Sigma_T) \ell'(\phi^{in}_{k,1}) - \Sigma_T W'_{\text{dw}}(\phi^{in}_{i,0}) - \Sigma_T W'_{\text{dw}}(\phi^{in}_{j,0}).
\end{align*}
As $\phi^{in}_{i,0} + \phi^{in}_{j,0} = 1$ we conclude $\ell'(\phi^{in}_{k,1}) = 0$ and with this $\phi^{in}_{k,1} \geq 0$. The asymptotic expansion of \eqref{rmodel5} for phase $j$ results in
\begin{align}
\begin{split}
\label{phiinterface}
0 &= \left(\Sigma_j - \Sigma_T\right)W'_{\text{dw}}(\phi^{in}_{j,0})  - \Sigma_T W'_{\text{dw}}(\phi^{in}_{i,0}) - \Sigma_T \ell'(\phi^{in}_{k,1}) - \Sigma_j \partial_{zz} \phi^{in}_{j,0} \\
&= \Sigma_j \left(W'_{\text{dw}}(\phi^j_0) - \partial_{zz} \phi^j_0 \right).
\end{split}
\end{align}
The matching condition \eqref{matching0} implies $\phi^{in}_{j,0}(-\infty)=0$ and $\phi^{in}_{j,0}(\infty)=1$. Following from the definition of $\Gamma_{ij}$ in \eqref{eq_gamma} we also get $\phi^{in}_{j,0}(0)=\frac 12$. With this the solution to \eqref{phiinterface} is given by
\begin{align}
\label{eq_phi0_solution}
\phi^{in}_{j,0}(z) = \frac{1}{2}\left(1+\tanh (3 z) \right).
\end{align}
Note that if we multiply \eqref{phiinterface} by $\partial_z \phi^j_0$, integrate and use the matching conditions \eqref{matching0}, \eqref{matchingD0} we find the equipartition of energy
\begin{align}
\label{equipartition}
W_{\text{dw}}(\phi^{in}_{j,0}) = \frac{1}{2} \left(\partial_z \phi^{in}_{j,0}\right)^2.
\end{align}
The leading order expansion of the Ginzburg--Landau free energy \eqref{eq_ginzburg} reads
\begin{align*}
f(\Phiv^{in},\nabla \Phiv^{in}) &= \eps^{-1} W(\Phiv^{in}_{0}) + \eps^{-1} \Sigma_i  \frac{1}{2} (\partial_z \phi^{in}_{i,0})^2 + \eps^{-1} \Sigma_j  \frac{1}{2} (\partial_z \phi^{in}_{j,0})^2 + O(1) \\
&= \eps^{-1} (\Sigma_i + \Sigma_j) \left( W_{\text{dw}}(\phi^{in}_{j,0}) +  \frac{1}{2}  (\partial_z \phi^{in}_{j,0})^2 \right) + O(1) .
\end{align*}
We can define the surface energy $\sigma_{ij}$ as the integral over the Ginzburg--Landau free energy, that is
\begin{align}
\begin{split}
\label{modelsigma}
\sigma_{ij} :&= \int_{-\infty}^\infty (\Sigma_i + \Sigma_j) \left( W_{\text{dw}}(\phi^{in}_{j,0}) +  \frac{1}{2}  (\partial_z \phi^{in}_{j,0})^2 \right) \diff z \\
&= (\Sigma_i + \Sigma_j) \int_{-\infty}^\infty \left(\partial_z \phi^{in}_{j,0}\right)^2 \diff z \\
&= \Sigma_i + \Sigma_j
\end{split}
\end{align}
where we have used \eqref{equipartition} and an explicit calculation after inserting \eqref{eq_phi0_solution}.
\paragraph{Expansion of \eqref{rmodel1}, $O(\eps^{-1})$:}
Using the transformation rules, the leading order is
\begin{align*}
\partial_z (\phi^{in}_{f,0} \vv^{in}_0 )\cdot \nv = 0.
\end{align*}
Note that with the considerations above, we have $\phi^{in}_{f,0} = \phi^{in}_{1,0} + \phi^{in}_{2,0} > 0$ across all interfaces. For $\Gamma_{12}$ we find by integrating and using matching condition \eqref{matching0}
\begin{align}
\label{vn_bc}
\vv^{in}_0(z) \cdot \nv = \vv^o_0(t,\sv_+) \cdot \nv = \vv^o_0(t,\sv_-)  \cdot \nv \quad \forall z\in(-\infty,\infty),
\end{align}
while for $\Gamma_{13}$ and $\Gamma_{23}$ we find with $\phi^{in}_{f,0}(\infty)=0$
\begin{align}
\label{vn_bc2}
\vv^{in}_0(z) \cdot \nv = \vv^o_0(t,\sv_-) \cdot \nv = 0   \quad \forall z\in(-\infty,\infty).
\end{align}
This is equation \eqref{eq_solid_fluid_v} of the sharp interface formulation.

\paragraph{Expansion of \eqref{rmodel3}, $O(\eps^{-2})$:}
We only consider the cases of $\Gamma_{12}$ and $\Gamma_{13}$. Then $\tilde \phi^{in}_c = \phi^{in}_{1,0} + O(\eps)$ and $\phi^{in}_{1,0}>0$. We note that $R_c$ is of order $\eps^{-1}$ . Therefore we have in leading order
\begin{align*}
\partial_{z} ( \phi^{in}_{1,0} \partial_z c^{in}_0) = 0.
\end{align*}
Then the matching conditions \eqref{matching0}, \eqref{matchingD0} at $z=-\infty$ imply
\begin{align}
\label{eq_c0_const}
c^{in}_0(z) = c^o_0(\sv_-) \quad \forall z\in(-\infty,\infty).
\end{align}

\paragraph{Expansion of \eqref{rmodel2}, $O(\eps^{-2})$:}
Again, note that $R_f$ and $\tilde \Sv$ are of order $\eps^{-1}$, so
\begin{align*}
0 &= \partial_z \Big(\gamma(\Phiv^{in}_0) ((\partial_{z} \vv^{in}_0) \otimes \nv + \nv \otimes (\partial_{z} \vv^{in}_0) )\Big) \nv \\
&= \partial_z (\gamma(\Phiv^{in}_0) \partial_{z} \vv^{in}_0).
\end{align*}
To get to the second line we have used that \eqref{vn_bc}, \eqref{vn_bc2} imply $\partial_z \vv^{in}_0\cdot \nv = 0$. Integrating and using the matching condition \eqref{matchingD0} gives
\begin{align*}
0 = \gamma(\Phiv^{in}_0) \partial_{z} \vv^{in}_0.
\end{align*}
As $\gamma(\Phiv^{in}_0)$ is positive, we find 
\begin{align}
\label{dzv}
\partial_z \vv^{in}_0 = 0.
\end{align}
With matching condition \eqref{matching0} we conclude
\begin{align}
\label{v_bc_1}
\vv^{in}_0(z) = \vv^o_0(t,\sv_-) = \vv^o_0(t,\sv_+) \quad \forall z\in(-\infty,\infty).
\end{align}
This equation is the continuity of $\vv$, given by \eqref{eq_v_continuous}, in the sharp interface formulation.

\paragraph{Expansion of \eqref{rmodel4}, \eqref{rmodel4b}, $O(\eps^{-1})$:}
We consider the interface $\Gamma_{13}$. We obtain for the phase field equations \eqref{rmodel4b} for phase $3$ and \eqref{rmodel4} for phase $1$ in leading order
\begin{align*}
- \nu \partial_z \phi^{in}_{3,0} - \frac{1}{\Sigma_3}\partial_{zz} \mu^{in}_{3,0} = q(\Phiv^{in}_{0}) \left( r(c^{in}_0) + \tilde \alpha_0 \mu^{in}_{1,0} - \tilde \alpha_0 \mu^{in}_{3,0}\right),\\
- \nu \partial_z \phi^{in}_{1,0} - \frac{1}{\Sigma_1}\partial_{zz} \mu^{in}_{1,0} = -q(\Phiv^{in}_{0}) \left( r(c^{in}_0) + \tilde \alpha_0 \mu^{in}_{1,0} - \tilde \alpha_0 \mu^{in}_{3,0}\right).
\end{align*}
Note that by \eqref{rmodel_alpha} we have $\tilde \alpha = \alpha + O(\eps)$. With the notation $\mu_{3-1} := \mu^{in}_{3,0} - \mu^{in}_{1,0}$ we get
\begin{align*}
- \nu \Sigma_3 \partial_z \phi^{in}_{3,0} + \nu \Sigma_1 \partial_z \phi^{in}_{1,0} - \partial_{zz} \mu_{3-1} = (\Sigma_1 + \Sigma_3) q(\Phiv^{in}_{0}) \left( r(c^{in}_0) - \alpha \mu_{3-1}\right).
\end{align*}
As there are no third-phase contributions in leading order we have $\partial_z \phi^{in}_{1,0} + \partial_z \phi^{in}_{3,0}= 0$. By construction of $q$ (see Remark \ref{Remark_q}) and the equipartition of energy \eqref{equipartition} it holds $q(\Phiv^{in}_0) = \partial_z \phi^{in}_{3,0}$. We have
\begin{align}
\label{eq_mu3-1}
- \nu (\Sigma_3+\Sigma_1) \partial_z \phi^{in}_{3,0} - \partial_{zz} \mu_{3-1} = (\Sigma_1 + \Sigma_3) \partial_z \phi^{in}_{3,0} \left( r(c^{in}_0) - \alpha \mu_{3-1}\right).
\end{align}
We interpret \eqref{eq_mu3-1} as an ordinary differential equation for $\mu_{3-1}$. From the matching condition \eqref{matchingD0} we get the asymptotic boundary conditions $\partial_z \mu_{3-1}(-\infty) = \partial_z \mu_{3-1}(\infty) = 0$. Now we need to distinguish between the cases $\alpha = 0$ and $\alpha > 0$.

For $\alpha=0$, integrating over equation \eqref{eq_mu3-1} results in
\begin{align}
\label{eq_nu_r}
- \nu (\Sigma_3+\Sigma_1) = (\Sigma_3+\Sigma_1) r(c^{in}_0).
\end{align}
This is a compatibility condition for the existence of solutions to \eqref{eq_mu3-1} (note that $r(c^{in}_0)$ is constant because of \eqref{eq_c0_const}). When fulfilled, any constant function is a solution to \eqref{eq_mu3-1}.

For $\alpha>0$ consider first the homogeneous part of \eqref{eq_mu3-1}, that is
\begin{align*}
\Big(-\partial_{zz} + (\Sigma_1 + \Sigma_3)  (\partial_z \phi^{in}_{3,0}) \alpha \Big) \mu = 0.
\end{align*}
This allows only for the solution $\mu=0$. Therefore the unique solution to \eqref{eq_mu3-1} is given by
\begin{align*}
\mu_{3-1}(z) = \frac{1}{\alpha} (\nu + r(c^{in}_0)).
\end{align*}
Rearranging this, we can express the velocity of the interface as
\begin{align}
\label{eq_nu}
\nu = \alpha\mu_{3-1}(z) - r(c^{in}_0).
\end{align}
Note that this expression also holds true for the case $\alpha = 0$, following from \eqref{eq_nu_r}.

\paragraph{Expansion of \eqref{rmodel4}, $O(\eps^{-1})$:} Consider $\Gamma_{23}$. Arguing similar as above we find that the leading order expansion
\begin{align*}
- \nu \partial_z \phi^{in}_{2,0} - \frac{1}{\Sigma_1}\partial_{zz} \mu^{in}_{2,0} = 0
\end{align*}
allows for each constant function $\mu^{in}_{2,0}$ as a solution, as long as the compatibility condition
\begin{align}
\label{eq_nu_0}
\nu = 0
\end{align}
is fulfilled. With the same argument applied to the equation for $\phi_1$ we conclude $\mu^{in}_{1,0}$ to be constant. 

The compatibility condition \eqref{eq_nu_0} corresponds to \eqref{eq_s_0} in the sharp interface formulation.

\paragraph{Expansion of \eqref{rmodel4}, $O(\eps^{-1})$:}
Consider $\Gamma_{12}$. Analogous to the result above we get the compatibility condition
\begin{align}
\label{eqNuV}
\nu  = \vv^{in}_0 \cdot \nv,
\end{align}
and all constant functions $\mu^{in}_{1,0}$, $\mu^{in}_{2,0}$ are solutions. 

The compatibility condition \eqref{eqNuV} corresponds to \eqref{eq_s_v} in the sharp interface formulation.

\subsection{Solution of Inner Expansions, First Order}
\label{sInnerFirst}

\paragraph{Expansion of \eqref{rmodel3}, $O(\eps^{-1})$:}
We only consider the interfaces $\Gamma_{12}$ and $\Gamma_{13}$. Substituting \eqref{modelR}, \eqref{modelrR} and the inner expansions we obtain with \eqref{eq_c0_const}
\begin{align}
\begin{split}
\label{eqRM3FirstOrder}
& - \nu \partial_z(\phi^{in}_{1,0} c^{in}_0) +  \partial_z(\phi^{in}_{1,0} c^{in}_0 \vv^{in}_0)\cdot \nv - \frac{1}{\Sigma_1}\partial_z(c^{in}_0\partial_z \mu^{in}_{1,0}) \\
&\qquad = D \partial_z(\phi^{in}_{1,0}\partial_z c^{in}_1) -c^\ast  q(\Phiv^{in}_0) (r(c^{in}_0) + \alpha \mu^{in}_{1,0}-\alpha \mu^{in}_{3,0}).
\end{split}
\end{align}

In the case of the fluid-solid interface $\Gamma_{13}$ we have $\vv^{in}_0 \cdot \nv = 0$ and $q(\Phiv^{in}_0) = \sqrt{2 W_\text{dw}(\phi_1)} = \partial_z \phi^{in}_{3,0}$, so by integrating we conclude
\begin{align*}
 \nu c^{in}_0 = -D \partial_z c^{in}_1(-\infty) - c^\ast (r(c^{in}_0) + \alpha \mu^{in}_{1,0}-\alpha \mu^{in}_{3,0}).
\end{align*}
With \eqref{eq_nu} and matching condition \eqref{matchingD1} we get
\begin{align}
\nu (c^\ast - c^{in}_0) = D \nabla c^o_0(t,\sv_-) \cdot \nv,
\end{align}
which describes \eqref{eq_rh} of the sharp interface formulation.

If we consider the fluid-fluid interface $\Gamma_{12}$ instead, we have $q(\Phiv^{in}_0) = 0$ and conclude from \eqref{eqRM3FirstOrder}
\begin{align*}
 c^{in}_0 \Big( (\vv^{in}_0\cdot \nv -\nu) \partial_z\phi^{in}_{1,0}  - \frac{1}{\Sigma_1}\partial_{zz} \mu^{in}_{1,0} \Big) = D \partial_z(\phi^{in}_{1,0}\partial_z c^{in}_1).
\end{align*}
With \eqref{eqNuV} and by integrating and matching conditions \eqref{matchingD0}, \eqref{matchingD1}
\begin{align*}
0 = \nabla c^o_0(t,\sv_-) \cdot \nv,
\end{align*}
which corresponds to \eqref{eq_rh2} of the sharp interface formulation.

\paragraph{Expansion of \eqref{rmodel5}, $O(1)$:}
At an interface $\Gamma_{ij}$, consider the difference $\mu_i - \mu_j$. With \eqref{rmodel5} we can write
\begin{align*}
\mu_i - \mu_j = \frac{1}{\eps}\left(\Sigma_i W'_{\text{dw}}(\phi_i) - \Sigma_j W'_{\text{dw}}(\phi_j)\right) - \eps \Sigma_i \Delta \phi_i + \eps \Sigma_j \Delta \phi_j .
\end{align*}
As $0<\phi^{in}_{i,0}<1$ and $\phi^{in}_{i,0} + \phi^{in}_{j,0} = 1$ the $O(1)$ terms of this expansion are given by
\begin{align*}
\mu^{in}_{i,0} - \mu^{in}_{j,0} &= \Sigma_i W''_{\text{dw}}(\phi^{in}_{i,0})\phi^{in}_{i,1} - \Sigma_j W''_{\text{dw}}(\phi^{in}_{j,0})\phi^{in}_{j,1} \\
&\qquad - \Sigma_i \left(-\kappa \partial_z \phi^{in}_{i,0} + \partial_{zz} \phi^{in}_{i,1}\right) + \Sigma_j \left(-\kappa \partial_z \phi^{in}_{j,0} + \partial_{zz} \phi^{in}_{j,1}\right) \\
&= \left(W''_{\text{dw}}(\phi^{in}_{i,0}) - \partial_{zz}\right)\left( \Sigma_i\phi^{in}_{i,1} - \Sigma_j\phi^{in}_{i,1} \right) + (\Sigma_i + \Sigma_j)\kappa \partial_z \phi^{in}_{i,0}.
\end{align*}
We interpret this as a differential equation with $\Sigma_i\phi^{in}_{i,1} - \Sigma_j\phi^{in}_{i,1}$ as the function to solve for. By the Fredholm alternative, this differential equation has a solution if and only if 
\begin{align*}
\int_{-\infty}^\infty (\mu^{in}_{i,0} - \mu^{in}_{j,0}) \partial_z \phi^{in}_{i,0} \diff z = \int_{-\infty}^\infty (\Sigma_i + \Sigma_j)\kappa (\partial_z \phi^{in}_{i,0})^2 \diff z.
\end{align*}
Using the definition of $\sigma_{ij}$ in \eqref{modelsigma} and the fact that $\mu^{in}_{i,0} - \mu^{in}_{j,0}$ does not depend on $z$ we find
\begin{align}
\label{eqMuSigma}
\mu^{in}_{j,0} - \mu^{in}_{i,0} = (\Sigma_i + \Sigma_j) \kappa = \sigma_{ij} \kappa.
\end{align}
With this the compatibility condition \eqref{eq_nu} for the reactive interface $\Gamma_{13}$ reads
\begin{align}
\nu = \alpha\sigma_{13} \kappa - r(c^{in}_0),
\end{align}
which is the interface condition \eqref{eq_s_r_alpha} of the sharp interface formulation.

\paragraph{Expansion of \eqref{rmodel2}$ \cdot \nv$, $O(\eps^{-1})$:}
Let us look at the case of the fluid-fluid interface $\Gamma_{12}$. Condition \eqref{dzv} simplifies the analysis. In particular, we have
\begin{align}
\begin{split}
\label{eq_firstOrderStress}
\nabla \cdot (2 \tilde \gamma (\Phiv^{in}) \nabla^s \vv) 
= \frac{1}{\eps} \partial_z \Big(\gamma(\Phiv^{in}_0) \Big(&(\partial_{z} \vv^{in}_1) \otimes \nv + \nabla_\Gamma \vv^{in}_0  \\
& + \nv \otimes (\partial_{z} \vv^{in}_1) + (\nabla_\Gamma \vv^{in}_0)^t \Big)\Big) \nv + O(1).
\end{split}
\end{align} 
With this, equation \eqref{rmodel2} at order $O(\eps^{-1})$ reads as
\begin{align*}
&-\nu \partial_z (\rho^{in}_{f,0}  \vv^{in}_0 ) + (\partial_z \rho^{in}_{f,0}) (\nv \cdot \vv^{in}_0) \vv^{in}_0 + \partial_z p^{in}_0 \nv  + \mu^{in}_{2,0} \partial_z \phi^{in}_{1,0} \nv + \mu^{in}_{1,0} \partial_z \phi^{in}_{2,0} \nv \\
&\qquad= \partial_z \Big(\gamma(\Phiv^{in}_0) ((\partial_{z} \vv^{in}_1) \otimes \nv + \nabla_\Gamma \vv^{in}_0+ \nv \otimes (\partial_{z} \vv^{in}_1) + (\nabla_\Gamma \vv^{in}_0)^t )\Big)\nv.
\end{align*}
With \eqref{eqNuV} the first two terms cancel out. Using the fact that $\mu^{in}_{1,0}$ and $\mu^{in}_{2,0}$ are constant, integrating over $z$ and applying matching condition \eqref{matchingJacobian0} yields
\begin{align*}
\jump{p}\nv + \mu^{in}_{1,0} \nv - \mu^{in}_{2,0} \nv = \jump{\gamma(\Phiv^o_0) (\nabla \vv^o_0 + (\nabla \vv^o_0)^t) } \nv.
\end{align*}
We use \eqref{eqMuSigma} to conclude the interface condition
\begin{align*}
\jump{pI - 2 \gamma \nabla^s \vv^o_0} \nv = \sigma_{12} \kappa \nv,
\end{align*}
corresponding to \eqref{eq_jump_stress} of the sharp interface formulation.

\paragraph{Expansion of \eqref{rmodel2}$ \cdot \tauv$, $O(\eps^{-1})$:}
Finally, for the fluid-solid interface $\Gamma_{13}$ and $\Gamma_{23}$, we again use conditions \eqref{dzv} and \eqref{eq_firstOrderStress}. Note that
\begin{align*}
(\nabla_\Gamma \vv^{in}_0) \nv = 0,
\end{align*}
as the surface gradient is perpendicular to the surface normal, and
\begin{align*}
(\nabla_\Gamma \vv^{in}_0)^t \nv = \nabla_\Gamma (\vv^{in}_0 \cdot \nv) = 0.
\end{align*}
With this, equation \eqref{rmodel2} at order $O(\eps^{-1})$ reads as
\begin{align*}
&-\nu \partial_z (\rho^{in}_{f,0}  \vv^{in}_0 ) + \phi^{in}_{f,0} \partial_z p^{in}_0 \nv  -\mu^{in}_{2,0} \phi^{in}_{f,0} \partial_z \left(\frac{\phi^{in}_{1,0}}{\tilde \phi^{in}_{f,0}}\right) \nv -\mu^{in}_{1,0} \phi^{in}_{f,0} \partial_z \left(\frac{\phi^{in}_{2,0}}{\tilde \phi^{in}_{f,0}} \right)\nv \\
&\qquad= \partial_z \Big(\gamma(\Phiv^{in}_0) ((\partial_{z} \vv^{in}_1) \otimes \nv + \nv \otimes (\partial_{z} \vv^{in}_1) \Big)\nv + \frac{1}{2}\rho_1 \vv^{in}_0 q(\Phiv^{in}_0)\nu,
\end{align*}
where we used \eqref{eq_nu}, \eqref{eq_nu_0} for the reaction term.
We only consider the tangential component of this equation. That is, we multiply with an arbitrary vector $\tauv \perp \nv$ and get
\begin{align*}
&-\nu \partial_z (\rho^{in}_{f,0}  \vv^{in}_0\cdot \tauv ) = \partial_z \Big(\gamma(\Phiv^{in}_0) \partial_{z} \vv^{in}_1 \cdot \tauv \Big) + \frac{1}{2} \nu \rho_1  \partial_z \phi^{in}_{3,0} \vv^{in}_0\cdot \tauv.
\end{align*}
Integrating and using \eqref{matching0} and \eqref{matchingD1} we get the interface condition
\begin{align}
\label{eq_stress_jump}
\frac{1}{2}\nu \rho_1 \vv^o_0 \cdot \tauv = \jump{\gamma \partial_\nv (\vv^o_0 \cdot \tauv)},
\end{align}
which is condition \eqref{eq_solid_fluid_v2} of the sharp interface formulation for $\Gamma_{13}$ and  $\Gamma_{23}$.

We remark that the left hand side term in \eqref{eq_stress_jump} exists due to the fact that the $\delta$-$2f1s$-model preserves kinetic energy instead of momentum during precipitation and dissolution.
\begin{figure}[!bp]
\centering
 \begin{tikzpicture}
\begin{scope}[scale = 3.5]

\def\vecpos{0.}
\def\textpos{0.8}
\def\s3{1.7320508}
\def\vdist{0.0em}
\def\vlength{2.5em}

\draw (0,0) to[out=125,in=-5] 
		node [fill=white,pos=\textpos] {$\Gamma_{12}$} 
		node (n01) [shape=coordinate,sloped,pos=\vecpos,xshift=\vlength/5,yshift=\vdist] {}
		node (n03) [shape=coordinate,sloped,pos=\vecpos,xshift=\vlength,  yshift=\vdist]{}
		(-0.8,0.3);
\draw (0,0) to[out=-120,in=0] 
		node [fill=white,pos=\textpos] {$\Gamma_{13}$} 
		node (n11) [shape=coordinate,sloped,pos=\vecpos,xshift=\vlength/5,yshift=-\vdist] {}
		node (n13) [shape=coordinate,sloped,pos=\vecpos,xshift=\vlength,  yshift=-\vdist]{}
		(-0.8,-0.3);
\draw (0,0) to[out=20,in=170] 
		node [fill=white,pos=\textpos] {$\Gamma_{23}$} 
		node (n21) [shape=coordinate,sloped,pos=\vecpos,xshift=-\vlength/5,yshift=\vdist] {}
		node (n23) [shape=coordinate,sloped,pos=\vecpos,xshift=-\vlength,  yshift=\vdist]{}
		(0.8,0.1);

\draw[-stealth] (n01) --  (n03) node[right] {$\tauv_{12}$};

\draw[-stealth] (n11) --  (n13) node[right] {$\tauv_{13}$};

\draw[-stealth] (n21) --  (n23) node[above] {$\tauv_{23}$};

\end{scope}
\end{tikzpicture}
 \caption{Vectors at the triple junction}
  \label{Figure_tp2}
\end{figure}
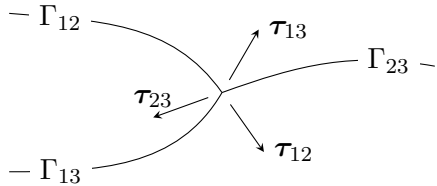
\begin{remark}
Considering the normal component of \eqref{rmodel2} at a fluid-solid interface leads to
\begin{align*}
\phi^{in}_{f,0} \partial_z \left( p^{in}_0 -\mu^{in}_{2,0} \frac{\phi^{in}_{1,0}}{\tilde \phi^{in}_{f,0}} -\mu^{in}_{1,0} \frac{\phi^{in}_{2,0}}{\tilde \phi^{in}_{f,0}}\right) = \partial_z \Big(2 \gamma(\Phiv^{in}_0) \partial_{z} \vv^{in}_1 \cdot \nv \Big).
\end{align*}
As we do not expect the right hand side to vanish, $\phi^{in}_{f,0} \partial_z p^{in}_0$ has to balance this term. That means that in the region where $\phi^{in}_{f,0}$ gets small, the assumption of $\partial_z p^{in}_0 = O(1)$ is no longer valid. Indeed, numerical simulations show that $p$ can oscillate in the solid part of a fluid-solid interface.
\end{remark}

\subsection{Triple Point Expansions}

As we have three bulk phases $\Phiv^o_0=\ev_1,\ev_2,\ev_3$ there are regions where these three phases meet. In the two-dimensional case we assume that the three phases meet at distinct points, called triple points. In the three-dimensional case we assume they meet at distinct lines, called triple lines.

In two dimensions the analysis of the triple points 
\begin{align*}
\Gamma_{123}(t) = \set{\xv \in \Omega: \Phiv(t,\xv) = (1/3,1/3,1/3)^t }
\end{align*}
can be done exactly as in Ref.~\cite{Bronsard1993,Garcke98}. For this one introduces local coordinates around a $ \hat \xv \in \Gamma_{123}$ 
and assumes that 
solutions to the $\delta$-$2f1s$-model can be written in terms of \textbf{triple point expansions} in these local coordinates. After matching the triple point expansions with the inner expansions of the three interfaces $\Gamma_{12}, \Gamma_{13}, \Gamma_{23}$ one obtains in leading order the condition
\begin{align}
\label{eq_tp_sigma}
0 &= \sum_{ij\in\set{12,13,23}} \sigma_{ij} \tauv_{ij},
\end{align}
where $\tauv_{ij}$ is the tangential unit vector of $\Gamma_{ij}$ at $\hat \xv$, as shown in Figure \ref{Figure_tp2}.

Condition \eqref{eq_tp_sigma} is equivalent to the contact angle condition \eqref{eq_contactangle} in the sharp interface formulation.

For the three-dimensional case, the analysis of the triple lines can be done exactly as in Ref.~\cite{stinner18pre}. We recover \eqref{eq_contactangle} on the plane perpendicular to the triple line.

\bibliographystyle{siamplain}
\bibliography{references}
\end{document}